\documentclass{amsart}

\pdfoutput=1

\usepackage[utf8]{inputenc}
\usepackage[english]{babel}
\usepackage{amssymb}
\usepackage[colorlinks]{hyperref}
\usepackage{enumitem}
\usepackage{etoolbox}
\usepackage{braket} 
\usepackage{faktor}
\usepackage{mathtools}
\usepackage{stackrel}
\usepackage{verbatim}
\usepackage{tikz-cd}
\usepackage{silence}
\RequirePackage{caption}
\RequirePackage{subcaption}

\usepackage[english]{cleveref}

\newtheorem{theorem}{Theorem}[section]
\newtheorem*{theo*}{Theorem}
\newtheorem{lemma}[theorem]{Lemma}
\newtheorem{proposition}[theorem]{Proposition}
\newtheorem{corollary}[theorem]{Corollary}
\theoremstyle{definition}
\newtheorem{definition}[theorem]{Definition}
\newtheorem{remark}[theorem]{Remark}

\newtheorem{conjecture}[theorem]{Conjecture}

\newcommand{\Z}{\mathbb{Z}}
\newcommand{\N}{\mathbb{N}}
\newcommand{\Q}{\mathbb{Q}}
\newcommand{\C}{\mathbb{C}}

\newcommand{\mG}{\mathcal{G}}

\DeclareMathOperator{\op}{op}

\DeclareMathOperator{\sgn}{sgn}

\DeclareMathOperator{\DP}{DP} 
\newcommand{\DPi}{\DP_{(a)}} 
\DeclareMathOperator{\PD}{PD} 
\DeclareMathOperator{\CM}{CM} 
\DeclareMathOperator{\HL}{HL} 
\DeclareMathOperator{\HR}{HR} 
\DeclareMathOperator{\Ann}{Ann} 
\DeclareMathOperator{\id}{Id} 
\DeclareMathOperator{\In}{in} 
\DeclareMathOperator{\im}{Im} 
\DeclareMathOperator{\tr}{tr}
\DeclareMathOperator{\dd}{d} 
\DeclareMathOperator{\cd}{cd} 

\makeatletter
\newcommand*{\bigcdot}{%
  {\mathbin{\mathpalette\bigcdot@{}}}%
}
\newcommand*{\bigcdot@scalefactor}{.75}
\newcommand*{\bigcdot@widthfactor}{1.4}
\newcommand*{\bigcdot@}[2]{%
  \sbox0{$#1\vcenter{}$}
  \sbox2{$#1\cdot\m@th$}%
  \hbox to \bigcdot@widthfactor\wd2{%
    \hfil
    \raise\ht0\hbox{%
      \scalebox{\bigcdot@scalefactor}{%
        \lower\ht0\hbox{$#1\bullet\m@th$}%
      }%
    }%
    \hfil
  }%
}
\makeatother

\robustify{\bigcdot}

\allowdisplaybreaks

\makeatletter
\g@addto@macro{\UrlBreaks}{\UrlOrds}
\g@addto@macro{\UrlBreaks}{%
\do\/\do\d%
}
\makeatother

\makeatletter
\newcommand\thankssymb[1]{\textsuperscript{\@fnsymbol{#1}}}
\makeatother

\begin{document}
\title{Hodge Theory for Polymatroids}
\author[R. Pagaria]{Roberto Pagaria}
\thanks{The first author is supported by PRIN 2017YRA3LK}
\address{Roberto Pagaria \newline \textup{Università di Bologna, Dipartimento di Matematica}\\ Piazza di Porta San Donato 5 - 40126 Bologna\\ Italy.}
\email{roberto.pagaria@unibo.it}

\author[G. M. Pezzoli]{Gian Marco Pezzoli}
\address{Gian Marco Pezzoli \newline \textup{Università di Bologna, Dipartimento di Matematica}\\ Piazza di Porta San Donato 5 - 40126 Bologna\\ Italy.}
\email{gianmarco.pezzoli2@unibo.it}

\begin{abstract}
We construct a Leray model for a discrete polymatroid with arbitrary building set and we prove a generalized Goresky-MacPherson formula.
The first row of the model is the Chow ring of the polymatroid; we prove Poincaré duality, Hard Lefschetz, and Hodge-Riemann theorems for the Chow  ring.
Furthermore, we provide a relative Lefschetz decomposition with respect to the deletion of an element.
\end{abstract}

\maketitle

\tableofcontents

\section{Introduction}
Recently, long standing conjectures about log-concavity of graphs and matroids have been brilliantly solved by studying the Chow ring of matroids \cite{Huh,HK,Lenz,HW,AHK,SingularHodgeTheory,ADHlagrangian,tautological}.
A natural question is whether the corresponding results hold for polymatroids \cite[Question 1.5]{tautological}.

\emph{Discrete polymatroids} generalize arrangements of subspaces in the same way as matroids generalize hyperplane arrangements.
Moreover polymatroids codify invariants of hypergraphs as matroids do with graphs.
Polymatroids have application in linear programming \cite{Edmonds,optimization}, in code theory \cite{codes}, and in commutative algebra \cite{HerzogHibi}.
The base polytope of a polymatroids is a \emph{generalized permutahedron} with a vertex in the origin, this yields a criptmorphism between polymatroids and generalized permutahedra \cite{Postnikov,AguiarArdila}.

The easiest definition of polymatroid is a pair $P=(E,\cd)$ where $E$ is a finite ground set and $\cd \colon 2^E \to \N$ is a \emph{increasing submodular function}.
If $P$ is realized by a subspace arrangement, then $\cd$ is the codimension of the corresponding flat.

In the case of arrangements of subspaces, De Concini and Procesi \cite{DeConciniProcesi} have constructed a \emph{wonderful model}, i.e.\ a smooth projective variety that contains the  complement of the arrangement as open subset.
This wonderful model $Y_\mG$ is obtained from $\mathbb{P}^r$ by a sequence of blowups along some linear subspaces; this collection of subspaces is called \emph{building set} $\mG$.
The variety $Y_\mG$ is used for studying the complement of the subspace arrangement, by considering the Leray spectral sequence for the inclusion of the complement in the wonderful model $Y_\mG$.
The spectral sequence collapses at the third page yielding a \emph{Leray model} (also known as Morgan algebra \cite{Morgan}) for the rational homotopy type.

Inspired by the realizable case, we provide a combinatorial definition of \emph{building set} for polymatroids and we introduce a \emph{Leray model} $B(P,\mG)$ for a polymatroid with building set.
In the case of matroids the Leray model was recently studied by Bibby, Denham, and Feichtner \cite{BDF}.
The last combinatorial object that we need is the $\mG$\emph{-nested set complex} $n(P,\mG)$.
In the realizable case this complex remembers whether the intersection of the corresponding divisors in $Y_\mG$ is non-empty.

The problem of computing the cohomology of the complement of a subspace arrangement was solved by Goresky and MacPherson \cite{GoreskyMacPherson} and by De Concini and Procesi \cite{DeConciniProcesi} with different techniques. 
Goresky and MacPherson have used the stratified Morse theory to describe the additive cohomology with integer coefficients of the complement in terms of a poset.
De Concini and Procesi have provided the aforementioned Morgan algebra (alias the Leray model) that describes the ring structure of the cohomology with rational coefficients.
These two results were connected by Yuzvinsky \cite{SmallRatModel,RatModel} constructing a smaller rational model $\CM$, however this connection was found only for the maximal building set.

We extend these results to the non-realizable setting and to arbitrary building sets, see Theorems \ref{thm:main_CM} and \ref{thm:main_Yuzvinsky}, by using the \emph{critical monomial algebra} $\CM(P,\mG)$.
\begin{theo*}
The inclusion $\CM(P,\mG) \hookrightarrow B(P,\mG)$ is a quasi-isomorphism. Moreover
\[H^{\bigcdot}(B(P,\mG)) \cong H^{\bigcdot}(\CM(P,\mG)) \cong \bigoplus_{f \in L} \bigotimes_{g \in F} \tilde{H}_{2\cd(g)-2- \bigcdot} \Bigl(  n((\hat{0},g), \mG) \Bigr). \let\qedsymbol\openbox\qedhere\]
\end{theo*}
In the realizable case with maximal building set, the above decomposition specializes to the Goresky-MacPherson formula.

The Leray model contains a subalgebra $\DP(P,\mG)$ as the first row of the spectral sequence, we call this algebra the \emph{Chow ring} of the polymatroid.
For subspace arrangements, $\DP(P,\mG)$ is the cohomology (indeed the Chow ring) of the wonderful model $Y_\mG$.
The combinatorial Chow ring for matroids was studied by Feichtner and Yuzvinsky \cite{FeichtnerYuzvinsky} and later by Huh, Katz, and Adiprasito \cite{Huh,HK,AHK} and others.

We prove that the Chow ring $\DP(P,\mG)$ of a polymatroid  satisfies the K\"ahler package (see Theorems \ref{thm:Poinc_duality} and \ref{thm:HL_HR}).
\begin{theo*}
The algebra $\DP(P,\mG)$ satisfy the Poincaré duality property.
Moreover, there exists a simplicial cone $\Sigma_{P,\mG}$ contained in $\DP^1(P,\mG)$ such that for each $\ell \in \Sigma_{P,\mG}$ the Hard-Lefschetz theorem and the Hodge-Riemann relations hold.
\end{theo*}
We prove the above theorem using methods similar to ones in \cite{AHK}.
A second and easier proof of the K\"ahler package for matroid was given in \cite{semismall} using a semismall decomposition; the decomposition is the first step through the singular Hodge theory \cite{SingularHodgeTheory}.
In the realizable setting the decomposition is induced by a map between wonderful models that is semismall (for semismall maps in algebraic geometry see \cite{dCM02,dCM09}).
In the case of polymatroids the corresponding map is not semismall, hence we cannot deduce the Kahler package using this method.
However, we obtain a relative Lefschetz decomposition of the Chow ring (see Theorem \ref{thm:Lef_dec}).
\begin{theo*}
Let $\DP_{(a)}$ be the Chow ring for the polymatroids $P\setminus a$ where an element $a \in E$ is removed from the ground set.
The Chow ring $\DP(P,\mG)$ decomposes into irreducible $\DP_{(a)}$-modules as
\[\DP(P,\mG) = \DPi \oplus \bigoplus_{f \in S_a} \bigoplus_{k=1}^{n_f} x_f^k \DPi.\]
Moreover, these irreducibles are explicitly described by:
\[ x_f^k \DPi \cong \DP((P\setminus a)^{f\setminus a}, (\mG \setminus a)^{f\setminus a}) \otimes \DP(P_f, \mG_f)  [k]. \let\qedsymbol\openbox\qedhere\]
\end{theo*}

The \emph{reduced characteristic polynomial} of a polymatroid is defined by
\begin{equation} \label{eq:def_chi_intro}
    \overline{\chi}_P(\lambda)= \frac{\sum_{A\subseteq E} (-1)^{\lvert A \rvert} \lambda^{\cd(E)-\cd(A)}}{\lambda-1}.
\end{equation} 
As final step we relate the coefficients of the reduced characteristic polynomial to the Hodge-Riemann bilinear form (see Theorem \ref{lemma:chi_deg}).
In order to do that, we restrict to the case of the maximal building set and we fix an isomorphism $\deg \colon \DP^r(P,\mG_{\max}) \to \Q$.
\begin{theo*}
There exist elements $\alpha,\beta \in \DP^1(P,\mG_{\max})$ such that 
\[\overline{\chi}_P(\lambda) = \sum_{i=0}^{r} (-1)^{i} \deg (\alpha^i \beta^{r-i}) \lambda^i. \let\qedsymbol\openbox\qedhere\] 
\end{theo*}

The element $\alpha$ belongs to the closure of the $\sigma$-cone (morally it is nef), but in general $\beta$ is not in the closure of the ample cone.
Hence, the coefficients of the reduced characteristic polynomial do not form a log-concave sequence (see Remark \ref{remark:not_log_concave}). Indeed every finite sequence of non-positive integers can appear as a substring of the coefficients.

\subsection*{Techniques}
The techniques used for our proof are various and inspired by \cite{FeichtnerYuzvinsky,BDF,SmallRatModel,AHK,semismall}.

In \Cref{sect:model} we make use of Gr\"obner basis to give two explicit additive bases of the Leray model.
In \Cref{sect:GMP_formula}, by using algebraic Morse theory, we compute the cohomology of the Leray model generalizing the Goresky-MacPherson formula.
In \Cref{sect:Kahler}, we use an inductive procedure 
to prove the K\"ahler package.
The main difference with the previous methods is that we do not have partial building sets as in \cite{BDF} nor order filters as in \cite{AHK}.
Our induction is based on the cardinality of the building set, and the inductive step involved completely different polymatroids.
\Cref{sect:rel_Lef} is devoted to the proof of the relative Lefschetz decomposition using some lemmas from the previous sections.
The reduced characteristic polynomial is studied in \Cref{sect:char_poly}. 
We prove the claimed equality by showing that both polynomials satisfy the same recursion. In this proof we use the properties of the M\"obius function for posets.

Finally, \Cref{sect:example} contains an explicit and exhaustive example that illustrates our definitions and properties.

\subsection*{Acknowledgements}
The authors thank Corrado De Concini for holding a mini-course on log-concavity results at the winter school ``Geometry, Algebra and Combinatorics of Moduli Spaces and Configurations IV''.
The authors also thank the organizers of the conference for having created this opportunity.

\section{The Leray model} \label{sect:model}
For general references about polymatroids we suggest \cite{WelshBook}.

\begin{definition}\label{def:polymatroid}
A \textit{polymatroid} $P=(E,\cd)$ on the ground set $E$ is a \textit{codimension} function $\cd \colon 2^{E} \to \N$ satisfying:
\begin{enumerate}
    \item[(C1)] $\cd(\emptyset)=0$,
    \item[(C2)] if $A \subseteq B$, then $\cd(A)\leq \cd(B)$, and
    \item[(C3)] if $A,B \subseteq E$, then $\cd(A \cup B)+\cd(A \cap B) \leq \cd(A)+\cd(B).$
\end{enumerate}
A polymatroid is a \textit{matroid} if the codimension of singletons are either zero or one.
\end{definition}
The closure of a subset $A\subseteq E$ is the subset 
    \[\{a \in E \ | \ \cd(A \cup \{a\})=\cd(A)\}.\]
A \textit{flat} is a closed set and the collection of flats forms a lattice $L_P$, that we call the \textit{poset of flats}. 
We use the notation $\max (X)$ with $X \subseteq L_P$ for the set of maximal elements of $X$.

\begin{definition}[Geometric building set]
Let $P=(E, \cd)$ be a polymatroid and let $L$ be its lattice of flats. A subset $\mathcal{G}$ in $L \setminus \{\hat{0}\}$ is called a \textit{geometric building set} if for any $x \in L$ the morphism of lattices:
    \[
      \varphi_x \colon  \prod_{y \in \max(\mathcal{G}_{\leq x})} \ [\hat{0},y]  \to  [\hat{0},x]
    \] 
induced by the inclusions is an isomorphism and the equality 
    \[\cd(x) = \sum_{y \in \max(\mathcal{G}_{\leq x})} \cd(y)
    \]
holds.

We define $F(P,\mathcal{G},x)=\max (\mathcal{G}_{\le x})$ the set of $\mathcal{G}$\textit{-factors} of $x$.
\end{definition}

\begin{definition}[$\mG$-nested set complex]
A subset $S$ of $\mathcal{G}$ is called $\mG$-nested if, for any set of incomparable elements $x_1,\dots,x_t$ in $S$ of cardinality at least two, the join $x_1 \vee \cdots \vee x_t$ is not contained in $\mG$. The $\mG$-nested sets form an abstract simplicial complex $n(P,\mG)$. 
\end{definition}

We suggest to visualize a (realizable) polymatroid $(E,\cd)$ as a collection of linear subspaces $S_e$ for $e\in E$ in a fixed complex vector space $V$.
For each $A \subseteq E$, the codimension $\cd(A)$ is the codimension of the corresponding flat $\cap_{a \in A} S_a$.
The (geometric) building set $\mathcal{G}$ is a good choice of some flats to blow up, in order to obtain a wonderful model $Y_\mathcal{G}$ with some exceptional divisors $D_g \subset Y_\mathcal{G}$, $g \in \mathcal{G}$ indexed by $\mathcal{G}$.
A subset $S$ of $\mathcal{G}$ is $\mathcal{G}$-nested if and only if the corresponding divisors $\{D_W\}_{W\in S}$ have non empty intersection.

The following proposition summarizes the main properties of building and nested sets.

\begin{proposition} \label{prop:comb_nested_sets}
Let $P$ be a polymatroid with poset of flats $L$ and $\mathcal{G}$ be a building set. Then:
\begin{enumerate}
    \item For each $g \in \mathcal{G}$, $x \in L$ with $x \geq g$, there exists a unique $\mathcal{G}$-factor $p$ of $x$ such that $p\geq g$.
    \item If $g,h \in \mathcal{G}$ and $g\wedge h > \hat{0}$, then $g\vee h \in \mathcal{G}$.
    \item If $S$ is a $\mathcal{G}$-nested set, then the $\mathcal{G}$-factors of $\bigvee S$ are the maximal elements in $S$ (i.e.\ $F(P,\mG,\bigvee S) = \max(S)$).
    \item Let $S$ be a $\mathcal{G}$-nested set, the Hasse diagram of $S$ (as subset of $L$) is a forest. \let\qedsymbol\openbox\qedhere
\end{enumerate}
\end{proposition}
\begin{proof}
For $(1)$ see \cite[Proposition 2.5(1)]{FeichtnerKozlov}, for $(2)$ see \cite[Proposition 2.5.3(b)]{BDF}, and for $(3)$ see \cite[Proposition 2.8]{FeichtnerKozlov}.
In order to prove $(4)$ we suppose that the Hasse diagram $\Gamma_S$ of $S$ is not a forest.
Thus there exist two incomparable elements $g,h \in S$ and $t\in S$ such that $g \wedge h \geq t \in S$;
in particular $g \wedge h > \hat{0}$. By part $(2)$ we get that $g \vee h \in \mG$ but this contradicts the definition of nested set. 
Therefore $\Gamma_S$ is a forest. 
\end{proof}

Let $P=(E, \cd)$ be a polymatroid, $L$ be its poset of flats, and $\mG$ be a building set in $L$.
Let $\mathcal{R}(\mathcal{G})=\mathbb{Q}[e_g,x_g \mid g \in \mathcal{G}]$ be the bigraded commutative algebra with exterior generators $e_g$ in bidegree $(0,1)$ and commutative generators $x_g$ in bidegree $(2,0)$. 

This algebra is equipped with a differential $d$ of bidegree $(2,-1)$ defined on generators by $d(e_g)=x_g$, $d(x_g)=0$.
Fix a linear extension $\succ$ of the order on $\mathcal{G}$, this gives a reverse order among the $e$-variables and among the $x$-variables, i.e.\ $x_h \prec x_g$ and $e_h \prec e_g$ if and only if $h \succ g$.
We also set $x_g \prec e_h$ for each $g,h$.
The algebra $\mathcal{R}(\mathcal{G})$ has a monomial basis given by:
    \[
    e_T x_S^{b}:=e_{g_1} \cdots e_{g_t} x_{h_1}^{b_1} \cdots x_{h_s}^{b_s}
    \]
where $T=\{g_1,\dots, g_t\}$ with $g_i \in \mG$ satisfying $g_1 \prec g_2 \prec \cdots \prec g_t$, $S=\{h_1, \dots, h_s\}$ with $h_i \in \mG$ and $b=(b_1,\dots, b_s)$ is a $s$-tuple of positive integers.
We define the element:
    \[
    c_g=\sum_{\substack{h \in \mathcal{G} \\ h \geq g}} x_h.
    \]

\begin{definition}[The Leray model of a polymatroid] \label{def:B}
Let $I(\mG)$ be the ideal of $\mathcal{R}(\mG)$ generated by
\begin{enumerate}[label=(\roman*)]
    \item $e_T x_S$ whenever $S \cup T \notin n(P, \mathcal{G})$,
    \item $e_T x_S c_g^{b}$ whenever $S, T \subseteq \mathcal{G}$, $g \in \mathcal{G}$ and $b\geq \cd (g) - \cd (\bigvee(S\cup T)_{<g})$.
\end{enumerate}
The ideal $I(\mG)$ is preserved by the differential $\dd$, so the quotient
    \[B(P, \mG)=\faktor{\mathcal{R}(\mG)}{I(\mG)}\]
is a bigraded differential algebra, called the \textit{Leray model} of the polymatroid.
\end{definition}

In the realizable case, the Leray model $B(P,\mathcal{G})$ is the second page of the Leray spectral sequence for the natural inclusion 
\[V \setminus \cup_{e \in E} S_e \cong Y_\mathcal{G} \setminus \cup_{g \in \mathcal{G}} D_g \hookrightarrow Y_\mathcal{G}.\]
This spectral sequence collapses at the third page, hence it becomes a differential bigraded algebra also known as the Morgan algebra \cite{Morgan}.

\begin{remark}\label{remark:secondarelazione}
Let $e_T x_S c_g^b$ be a monomial of type $(ii)$ and let
    \[S'=S \cap \max(S \cup T)_{<g} \quad \hbox{and} \quad T'= \max(S \cup T)_{<g} \setminus S.\]
The monomial $e_{T'} x_{S'} c_g^b$ divides $e_T x_S c_g^b$.
Thus, when we consider a monomial of type $(ii)$ we can always assume that $S\cap T= \emptyset$, $S \cup T$ is an antichain and $\bigvee(S \cup T)<g$.
\end{remark}

\begin{theorem} 
\label{thm:grobner_basis}
The generators of type (i) and (ii) of the ideal $I(\mathcal{G})$ of Definition \ref{def:B} form a Gr\"obner basis with respect to the deg-lex order.
\end{theorem}
\begin{proof}
We adapt the method used in \cite[Theorem 2]{FeichtnerYuzvinsky} and in  \cite[Theorem 5.3.1]{BDF}.
We are fixing a linear extension of the order on $\mG$ with $x_{g} \prec e_{h}$ for each $g,h$.
We consider the deg-lex monomial order on $\mathcal{R}(\mG)$ and
we explicitly compute $S$-polynomials.
\begin{description}
    \item [Case (i)-(i)] 
    Since relations type $(i)$ are monomials the $S$-polynomials is zero.
    \item [Case (i)-(ii)]  Now we consider $f_1=e_T x_S$ of type $(i)$ and $f_2=e_A x_B c_g^b$ of type $(ii)$.
    We assume that $\bigvee (A \cup B) < g$ (see Remark \ref{remark:secondarelazione}). Let $U= T \cup A$, $V=B \cup S \smallsetminus \{g\}$, therefore the $S$-polynomial is
        \[S(f_1,f_2)=e_U x_V x_g^b-e_U x_V c_g^b=e_U x_V (x_g^b-c_g^b).\]
    If $g \notin S$, we have that $S \subseteq V$ and therefore
        \[S(f_1,f_2)= \pm e_{A\setminus T} e_T x_S x_{V \setminus S} (x_g^b-c_g^b)\]
    is divisible by $e_T x_S$. 
    
    Then, assume $g \in S$, since $S \cup T$ is not nested we have that $U \cup V \cup \{g\}$ is not nested. If $U \cup V$ is not nested, then $S(f_1,f_2)$ would be divisible by $e_U x_V$.
    
    So assume $U \cup V$ is nested, (thus we have $g \notin U \cup V$) since $U \cup V \cup \{g\}$ is not nested the $S$-polynomial modulo $e_U x_V x_g$ became
        \[S(f_1,f_2) \equiv e_U x_V \big ( \sum_{f>g} x_f \big)^b.\]
    The set $U \cup V \cup \{g\}$ contains a non trivial antichain $Y$ whose join $  \bigvee Y=y$ is in $\mG$ and $Y$ must contain $g$ since $U \cup V$ is nested; let $y'=\bigvee (Y \smallsetminus \{g\})$. We have
        \begin{align*}
        b &=\cd(g)-\cd \Bigl( \bigvee_{\substack{l \in A \cup B\\l<g}} l \Bigr) \\
        &\geq \cd(g \vee y')- \cd \Bigl( \bigvee_{\substack{l \in A \cup B\\l<g}} l \vee y' \Bigr) \\
        & \geq \cd(y)- \cd \Bigl( \bigvee_{\substack{l \in U \cup V\\l<y}} l \Bigr)=b'
        \end{align*}
    and so $e_U x_V c_y^b$ is a relation of type $(ii)$.
    We claim that modulo relations of type $(i)$  
        \[S(f_1,f_2) \equiv e_U x_V c_y^b.\]
    To obtain this, we show that if $f \in \mG$ with $f > g$ and $f \ngeq y$, then $U \cup V \cup \{f\}$ is not nested. 
    Suppose that $U\cup V \cup  \{f\}$ is nested and consider the antichain $Y'= \max (Y \smallsetminus \{g\} \cup \{f\}) \subseteq U \cup V \cup \{f\}$.
    The set $Y'$ is nested and by Proposition \ref{prop:comb_nested_sets} the $\mG$-factors of $\bigvee (Y \smallsetminus \{g\} \cup \{f\})$ are exactly the elements of $Y'=\{y_1,\dots,y_k,f\}$. We have
        \[\bigvee (Y \smallsetminus \{g\} \cup \{f\})=y' \vee f \geq y' \vee g=y.\]
    By definition of $\mG$-factor we have two cases:
        \begin{itemize}
            \item $y\leq y_i$ for a certain $i$. But this is impossible since $y_i<y$;
            \item $y \leq f$ contrary to the assumption $f\not \geq y$.
        \end{itemize}
    Thus, $U \cup V \cup \{f\}$ is not nested and $S(f_1,f_2)$ reduces to zero.
    \item [Case (ii)-(ii)]  Let $f_1=e_T x_S c_g^d$ and $f_2=e_A x_B c_h^f$ be two relations of type $(ii)$.
    We may assume $\bigvee(S \cup T)<g$ and $\bigvee(A \cup B)<h$ (see Remark \ref{remark:secondarelazione}).
    We have the following cases:

    First $g=h$ and $d \leq f$, then the $S$-polynomial is
        \[S(f_1,f_2)=e_{T \cup A} x_{S \cup B} c_g^d (x_g^{f-d}-c_g^{f-d});\]
    which is divisible by $e_T x_S c_g^d$.
    
    Second $g \neq h$, $g \notin B$, $h \notin S$, we also assume that $h \succ g$. The $S$-polynomial is
        \[S(f_1,f_2)=e_{T \cup A} x_{S \cup B}(x_h^f c_g^d-x_g^dc_h^f).\]
    Let $y=e_{T \cup A} x_{S \cup B} c_g^d (c_h^f-x_h^f)$, which is divisible by $f_1=e_T x_S c_g^d$ and has a leading term smaller or equal to that of $S(f_1,f_2)$. 
    The remainder
        \[S(f_1,f_2)+y=e_{T \cup A} x_{S \cup B}(c_g^d-x_g^d)c_h^f,\]
    is divisible by $f_2=e_A x_B c_h^f$, and reduces to zero.
    
    Finally, assume $g \neq h$ and $g \in B$, by Remark \ref{remark:secondarelazione} we must have $g \prec h$ and $h \notin S$.
    Let $U = T\cup A$ and $V=S\cup B \smallsetminus \{g\}$, the $S$-polynomial is
        \[S(f_1,f_2)=e_U x_V (x_h^f c_g^d-x_g^d c_h^f).\]
    Let $y=e_U x_V c_g^d (c_h^f-x_h^f)$, which is divisible by $f_1=e_T x_S c_g^d$ and has a leading term smaller or equal to that of $S(f_1,f_2)$.
    It remains to verify that
        \[S(f_1,f_2)+y=e_U x_V (c_g^d-x_g^d)c_h^f\]
    reduces to zero.
    First, through division by $f_2=e_A x_B c_h^f$, since $g \in B$ we have 
        \begin{equation} \label{eq:to_be_reduced}
         S(f_1,f_2)+y \equiv e_U x_V \Big( \sum_{k>g}x_k \Big)^d c_h^f.
        \end{equation}
    We claim that for any $k>g$, $k \not \geq h$ we have
        \[e_U x_V x_k c_h^f \equiv e_U x_V x_k c_{h \vee k}^f \equiv 0\]
    modulo relations of type $(i)$ and $(ii)$. 
    For the first claim, if $p \geq h$ but $p \ngeq h \vee k$ then $\{p,k\}$ is not nested 
    by Proposition \ref{prop:comb_nested_sets} and we can divide by the relation $x_h x_p$ of type $(i)$.
    The last claim follows since 
    $h \vee k \in \mG$ by Proposition \ref{prop:comb_nested_sets} and
        \begin{align*}
        f &\geq \cd(h)-\cd(\bigvee (A \cup B))\\
        &\geq \cd(h \vee k)-\cd(\bigvee (U \cup V \cup \{k\})).
        \end{align*}
    Therefore, the element in eq.\ \eqref{eq:to_be_reduced} reduces to
    \[S(f_1,f_2)+y \equiv e_U x_V c_h^{d+f}.\]
    Since $d+f \geq \cd(h)-\cd(\bigvee (U \cup V \cup \{k\}))$
    we may divide by $e_U x_V c_h^{d+f}$ and reduce to zero.
\end{description}
This completes the proof.
\end{proof}

\begin{corollary} 
\label{cor:base_additiva_e_x}
The algebra $B(P, \mathcal{G})$ has an additive basis given by the monomials $e_T x_S^{b}$ such that $S\cup T \in n(P, \mathcal{G})$
and $0<b(s)< \cd(s) - \cd (\bigvee (S\cup T)_{<s})$ for all $s\in S$.
\end{corollary}
\begin{proof}
An additive basis for the algebra $B(P, \mathcal{G})$ is given by all the monomials which are not divisible by the initial monomials of the Gr\"obner basis provided by Theorem \ref{thm:grobner_basis}. The proof follows immediately. 
\end{proof}

We now provide a second presentation for the algebra $B(P,\mathcal{G})$ using a different set of generators ($\tau_g$ and $\sigma_g$ for $g \in \mathcal{G}$). 
This second presentation is inspired by the work of Yuzvinsky \cite{RatModel,SmallRatModel} and in the matroidal case coincides with the simplicial presentation of Backman, Eur, and Simpson \cite{BacEurCon}.

\begin{theorem} 
\label{thm:new_gen}
The morphism 
\[ \varphi \colon \Lambda [\tau_g \mid g \in \mG] \otimes \mathbb{Q}[\sigma_g \mid g \in \mG] \to B(P,\mG) \]
defined by $\varphi(\tau_g)= \sum_{h \geq g} e_h$ and by $\varphi(\sigma_g)= \sum_{h \geq g} x_h$, is surjective with kernel generated by:
\begin{enumerate}[label=(\roman*)]
    \item $\prod_{t \in T} (\tau_t - \tau_g) \prod_{s \in S} (\sigma_s - \sigma_g)$ for $S\cup T$ a non-trivial antichain and $g= \bigvee (S\cup T) \in \mG$,
    \item $\prod_{t \in T} (\tau_t- \tau_g) \prod_{s \in S} (\sigma_s - \sigma_g) \sigma_g^b$ for $g \in \mG$ and $b=\cd (g) - \cd (\bigvee(S\cup T)_{<g})$. \let\qedsymbol\openbox\qedhere
\end{enumerate}
\end{theorem}

We will identify the elements $\tau_g, \sigma_g$ with their images in $B(P,\mG)$. 
In the realizable case the element $\sigma_g$ is the fundamental class of $D_g$, the total transform of the flat $g$. 
Analogously, $\tau_g$ is the sum of irreducible components of the total transform of the flat $g$.
The elements $\sigma_g$ can be also seen in the following way: consider the inclusion $Y_{\mG} \hookrightarrow \prod_{h\in \mG} \mathbb{P}^{\cd(h)-1}$ of \cite{DeConciniProcesi}, $\sigma_g$ is the pullback of the hyperplane class of the factor $\mathbb{P}^{\cd(g)-1}$.

Before the proof of Theorem \ref{thm:new_gen} we need a couple of technical lemmas.

\begin{lemma}\label{lemma_tecnico}
Let $g \in \mG$ and $S=\{s_1, \dots s_n \} \subset \mG$ such that $\bigvee S \leq g$, set $b=\cd(g)-\cd(\bigvee S)$.
Consider a set $A=\{a_1, \dots, a_n \} \subset \mG$ such that $a_i\geq s_i$ and $a_i \not \geq g$ for all $i=1,\dots, n$.
Then 
\[ y_A c_g^b=0, \]
where $y_{a_i}$ is equal to $e_{a_i}$ or $x_{a_i}$ and $y_A=y_{a_1} \cdots y_{a_n}$.
\end{lemma}
\begin{proof}
Define the element $h= \bigvee A \vee g$, we first prove the equality 
$y_A \sigma_g^b= y_A \sigma_h^b$ and then $y_A \sigma_h^b=0$.

We show that $h \in \mathcal{G}$.
Let $h' \in \mathcal{G}$ be the unique $\mathcal{G}$-factor of $h$ such that $h'\geq g$.
For each $a_i$ we have $h' \wedge a_i\geq s_i$ and so $a_i \vee h' \in \mathcal{G}$.
By maximality of $h'$ we have $a_i \leq h'$ for all $i$. Therefore $h=h' \in \mathcal{G}$.

Firstly, let $g'\in \mathcal{G}$ be any element such that $g'\geq g$ and $g' \not \geq h$.
Suppose that $A \cup \{g'\}$ is a $\mathcal{G}$-nested set.
Then the $\mathcal{G}$-factors of $h\vee g'$ are the maximal elements of $A \cup \{g'\}$ by Proposition \ref{prop:comb_nested_sets}.
So there exists an element in $A\cup \{g'\}$ bigger or equal to $h$, this is impossible since $g'\not \geq h$ and $a_i \not \geq g$.
It follows that $A \cup \{g'\}$ is not $\mathcal{G}$-nested and $y_Ax_{g'}=0$.

Finally, we show that $y_A \sigma_h^b=0$.
Indeed, $b\geq \cd(g)- \cd(g \wedge \bigvee A)$ which is  bigger than $\cd(h)- \cd(\bigvee A)$ by submodularity of $\cd$.
Applying the relations of type $(ii)$ in Definition \ref{def:B} we complete the proof.
\end{proof}

\begin{lemma}\label{lem:ker2}
The elements $\prod_{t \in T} (\tau_t - \tau_g) \prod_{s \in S} (\sigma_s - \sigma_g) \sigma_g^b$ for $g \in \mathcal{G}$,
and $b=\cd (g) - \cd (\bigvee(S\cup T)_{<g})$ belong to the kernel of $\varphi$.
\end{lemma}
\begin{proof}
From the argument of Remark \ref{remark:secondarelazione} we may assume that $S\cap T= \emptyset$, $S\sqcup T$ is an antichain, and $\bigvee (S\cup T) \leq g$.

We have
\[\varphi \Bigl( \prod_{t \in T} (\tau_t - \tau_g) \prod_{s \in S} (\sigma_s - \sigma_g) \sigma_g^b \Bigr) = \sum_{A,B} e_Ax_B \Bigl( \sum_{l\geq g} x_l \Bigr)^b,\]
where the sum is taken over the sets $A=(a_i)_i$ and $B=(b_j)_j$ such that $a_i \geq t_i$, $b_j \geq s_j$, $a_i \not \geq g$, and $b_j \not \geq g$.
Each term $e_Ax_B \Bigl( \sum_{l\geq g} x_l \Bigr)^b$ is zero by Lemma \ref{lemma_tecnico}.
\end{proof}

\begin{proof}[Proof of Theorem \ref{thm:new_gen}]
Let $\prec$ be a reverse linear extension of the order on $\mG$ with $x_{g} \prec e_{h}$ and $\sigma_{g} \prec \tau_{h}$ for each $g,h$.
Now we consider the basis formed respectively by the $\sigma_g$, $\tau_h$  and by the $x_g$,$e_h$ ordered with $\prec$; with respect of these two basis the matrix associated to the morphism $\varphi$ is upper unitriangular and therefore invertible.
It follows that the map $\varphi$ is surjective.

We want to prove that $\ker \varphi$ is generated by relations of type $(i)$ and $(ii)$ of Theorem \ref{thm:new_gen}.
From Lemma \ref{lem:ker2} we know that elements of the form $(ii)$ belong to $\ker \varphi$.
The relations $(i)$ are a particular case of relations $(ii)$ with $b=0$.
Let $J$ be the ideal generated by relations of type $(i)$ and $(ii)$, we denote also by $\In (J)$ the initial ideal of $J$.
It suffices to prove that
\[\dim \faktor{\mathcal{C}}{\In (J)} \leq  \dim \faktor{\mathcal{R}(\mG)}{\In(I(\mG))}\]
where $ \mathcal{C}=\Lambda [\tau_g \mid g \in \mG] \otimes \mathbb{Q}[\sigma_g \mid g \in \mG] $.
Let $K \subseteq \In(I)$ be the ideal generated by the leading monomial of relation of type $(i)$ and $(ii)$, since 
\[ \dim \faktor{\mathcal{C}}{K} \geq \dim \faktor{\mathcal{C}}{\In(J)}\]
it suffices to check that 
\begin{equation}
\label{eq:disuguaglianza}
    \dim \faktor{\mathcal{C}}{K} = \dim \faktor{\mathcal{R}(\mG)}{\In(I(\mG))}.
\end{equation}
The leading monomials of relation type $(i)$ are of the form $\tau_T \sigma_S$ where $S \cup T$ is not $\mG$-nested; the leading monomials of relation type $(ii)$ are of the form $\tau_T \sigma_S \sigma_g^{b}$ whenever $S, T \subseteq \mathcal{G}$, $g \in \mathcal{G}$ and $b=\cd (g) - \cd (\bigvee(S\cup T)_{<g})$.
The monomials in $\mathcal{C}$, which are not divisible by the these two type of leading monomials, are of the form $\tau_T \sigma_S^m$ with  $S\cup T \in n(P,\mG)$ and $0<m(s)< \cd(s) - \cd (\bigvee (S\cup T)_{<s})$ for all $s\in S$.
Hence eq.\ \eqref{eq:disuguaglianza} follows.
Since the map $\varphi$ is surjective it is also injective; and the initial ideal $\In(J)$ is equal to $K$.
Therefore, relations of type $(i)$ and $(ii)$ form a Gr\"obner basis for $\ker \varphi$.
\end{proof}

From the proof of Theorem \ref{thm:new_gen} we obtain also the following corollary:
\begin{corollary} \label{cor:base_additiva_sigma_tau}
The set of monomials $\tau_T \sigma_S^m$ with  $S\cup T \in n(P,\mG)$ and, for each $s\in S$, $0<m(s)< \cd(s) - \cd (\bigvee (S\cup T)_{<s})$ is an additive basis of $B(P,\mG)$.
\end{corollary}
See Section \ref{sect:example} for an example of the application of Corollary \ref{cor:base_additiva_e_x} and Corollary \ref{cor:base_additiva_sigma_tau}. 

\section{Generalized Goresky-MacPherson formula} \label{sect:GMP_formula}

In this section we generalize the Goresky-MacPherson formula (see \cite{GoreskyMacPherson}) to the non-realizable case and to arbitrary building set.
The choice of the minimal building set yields a significantly smaller nested set complex and it can be useful in practical computations.
Other generalizations of this formula can be found in \cite{EquivGoreskyMacPherson,deshpande,MociPagaria}.

\subsection{Critical monomials}
\begin{definition}
A \textit{standard monomial} $e_Tx_S$ (resp.\ $\tau_T \sigma_S$) is a monomial that appears in the basis given by Corollary \ref{cor:base_additiva_e_x} (resp.\ by Corollary \ref{cor:base_additiva_sigma_tau}).
\end{definition}

For any standard monomial $\tau_T \sigma_S^b$ we extend the function $b$ by setting $b(g)=0$ for $g \not \in S$.

\begin{definition}
Let $\tau_T \sigma_S^b$ be a standard monomial. An element $g \in \mathcal{G}$ is called \textit{critical} with respect to the monomial $\tau_T \sigma_S^b$ if $g \in T$ and $b(g)=\cd(g)-\cd(\bigvee(S \cup T)_{<g})-1$.
If every element of $S\cup T$ is critical with respect to $\tau_T \sigma_S^b$ then the monomial $\tau_T \sigma_S^b$ is called \textit{critical}.
\end{definition}
Notice that if the monomial $\tau_T \sigma_S^b$ is critical, then $S \subseteq T$ and so the critical monomial is uniquely determined by $T$.

\begin{definition}
The \textit{critical monomial} associated with $T\in n(P,\mathcal{G})$ is 
\[ c\mu (T) = \tau_T \sigma_S^b,\]
where $S=\{t \in T \mid \cd(t)-\cd (\bigvee(S_{<t}))>1\}$ and $b(s)=\cd(s)-\cd (\bigvee(T_{<s}))-1$ for all $s \in S$.
\end{definition}
 In Theorem \ref{thm:main_CM} we will prove that the linear span of critical monomials form a subcomplex (indeed a subalgebra) of $B(P,\mathcal{G})$.
Moreover, we will show that this subalgebra is quasi-isomorphic to the Leray model.
This first lemma implies that the span of critical monomials is a sub-complex.

\begin{lemma} 
\label{lemma:diff_critical}
For every critical monomial $c\mu(T)$ we have
\[ \dd (c\mu(T)) = \sum_{t \in T \setminus \max(T)} (-1)^{\lvert T_{\prec t}\rvert} c\mu(T\setminus \{t\}). \let\qedsymbol\openbox\qedhere\]
\end{lemma}
\begin{proof}
Let $c\mu(T)=\tau_T\sigma_S^b$, we have 
\begin{align*}
    \dd (c\mu(T)) &= \sum_{t \in T}(-1)^{\lvert T_{ \prec t}\rvert}  \tau_{T\setminus \{t\}} \sigma_S^b \sigma_t \\
    &= \sum_{t \in T \setminus \min(T)}(-1)^{\lvert T_{\prec t}\rvert}  \tau_{T\setminus \{t\}} \sigma_S^b \sigma_t,
\end{align*}
because if $t\in \min(T)$ then $b(t)=\cd(t)-1$ and so $\sigma_t^{\cd(t)}=0$.

Fix $t\in T\setminus \min(T)$, the set $R=\max(T_{<t})$ is nonempty. 
By using relation (ii) of Theorem \ref{thm:new_gen} and the fact that $\tau_t^2=0$, we have
\begin{align*}
     \tau_R \sigma_{R}^b \sigma_t^{b(t)+1} &= \sum_{r \in R} (-1)^{\lvert R_{ \prec r} \rvert} \tau_{R\setminus \{r\}}\tau_t \sigma_{R}^b \sigma_t^{b(t)+1} \\
     &= \sum_{r \in R}  (-1)^{\lvert R_{\prec r} \rvert} \tau_{R\setminus \{r\}}\tau_t \sigma_{R \setminus \{r\}}^b \sigma_t^{b(t)+b(r)+1},
\end{align*}
where in the last equality we used 
\[0=\tau_t (\sigma_r-\sigma_t) \sigma_t^{b(t)+1} \prod_{l \neq r} (\tau_l-\tau_t) = \tau_t \tau_{R \setminus \{r\}} (\sigma_r-\sigma_t) \sigma_t^{b(t)+1}.\]
Notice that $T\in n(P,\mathcal{G})$ implies $\cd(\bigvee R)= \cd (\bigvee(R\setminus\{r\})) + \cd (r)$ and $\cd(\bigvee(R\setminus\{r\}) \vee \bigvee (T_{<r}))=\cd(\bigvee(R\setminus\{r\}))+\cd(\bigvee (T_{<r}))$; so $b_{T\setminus \{r\}}(t)=b_T(t)+b_T(r)+1$.
Therefore 
\[ \tau_R \sigma_R^b \sigma_t^{b(t)+1} = \sum_{r \in R}  (-1)^{\lvert R_{ \prec r} \rvert} c\mu((R\setminus \{r\}) \cup \{t\})\]
and finally:
\begin{align*}
    \dd (c\mu(T)) &= \sum_{t \in T \setminus \min(T)} \sum_{r \in \max(T_{<t})} (-1)^{\lvert T_{ \prec r}\rvert} \mu(T\setminus \{r\})\\
     &= \sum_{r \in T \setminus \max(T)} (-1)^{\lvert T_{ \prec r}\rvert} \mu(T\setminus \{r\}),
\end{align*}
because $T$ is a forest by Proposition \ref{prop:comb_nested_sets}.
This conclude the proof.
\end{proof}

We want to apply algebraic Morse theory to the complex $B(P,\mG)$.
We refer to \cite{JW} for basic definitions and properties of algebraic Morse theory. 

We define the following \textit{matching} $\mathcal{M}$:
for each non-critical monomial $\tau_T \sigma_S^b$ let $g \in S\cup T$  be the smallest (with respect to $\prec$) non-critical element.
If $g$ belongs to $T$, then the pair $(\tau_T \sigma_S^b, \tau_{T\smallsetminus \{g\}} \sigma_{S}^{b} \sigma_g) $ is in $\mathcal{M}$.

The algebraic Morse theory, together with Lemma \ref{lemma:matching} and Proposition \ref{prop:Morse_matching}, implies that the complex of critical monomials is quasi-isomorphic to the Leray model.

\begin{lemma} \label{lemma:matching}
The set $\mathcal{M}$ is a matching.
Moreover, a monomial is critical if and only if it is critical for the matching $\mathcal{M}$.
\end{lemma}
\begin{proof}
We check that each non-critical monomial appears exactly once in $\mathcal{M}$ and that all monomials in $\mathcal{M}$ are non-critical.

By definition if the monomial $\tau_T \sigma_S^b$ appears in the first position in $\mathcal{M}$, it is non-critical.
Moreover $\tau_{T\smallsetminus \{g\}}\sigma_{S}^{b} \sigma_g$ is non-critical because $S \cup \{g\} \not \subseteq T\smallsetminus \{g\}$.
So every monomial in the matching is non-critical.

Vice versa, if $\tau_T \sigma_S^b$ is a non-critical monomial, let $g$ be the minimal non-critical element in $S\cup T$. If $g\in T$ then $\tau_T \sigma_S^b$ appears in the matching (in the first position).
Otherwise, $g\in S \setminus T$ so the monomial $\tau_T \tau_g \frac{\sigma_S^b}{\sigma_g}$ is basic and non-critical.
Notice that an element $f \in \mathcal{G}$  is critical for $\tau_T \sigma_S^b$ if and only if is critical for $\tau_T \tau_g \frac{\sigma_S^b}{\sigma_g}$.
Therefore the pair $(\tau_T \tau_g \frac{\sigma_S^b}{\sigma_g}, \tau_T \sigma_S^b)$ is in $\mathcal{M}$.
\end{proof}

\begin{definition}
Given a standard monomial $\tau_T \sigma_S^b$ we define $m(T,S,b)$ as the multiset $\{g^{a(g)} \mid g \in \mathcal{G}\}$ where $a(g)$ is the sum of the exponents of $\tau_g$ and $\sigma_g$ in the monomial $\tau_T \sigma_S^b$.
Moreover, we order these multisets lexicographically using the reverse order on $\mathcal{G}$.
\end{definition}
As an example, if $h < g$ then $h \prec g$ and $\{ h^2 \} \succ \{h, g\}$.

\begin{definition}
Let $G$ be the directed graph whose vertices are the standard monomials with a directed edge from $\tau_T \sigma_S^b$ to $\tau_{T'} \sigma_{S'}^{b'}$ if 
the monomial $\tau_{T'} \sigma_{S'}^{b'}$ appears with a nonzero coefficient in $\dd(\tau_T \sigma_S^b)$.

Let $G_{\mathcal{M}}$ be the directed graph $G$ with all directed edges in $\mathcal{M}$ reversed.
\end{definition}

\begin{proposition}\label{prop:Morse_matching}
    The matching $\mathcal{M}$ is a Morse matching.
\end{proposition}
\begin{proof}
We show that the graph $G_{\mathcal{M}}$ is acyclic.

Although $m$ is not a term order (because $m(\tau_g)=m(\sigma_g)$) it has the property that for any relation of Theorem \ref{thm:new_gen} 
\begin{equation}
\label{eq:inutile}
    \prod_{t \in T} (\tau_t- \tau_g) \prod_{s \in S} (\sigma_s - \sigma_g) \sigma_g^b
\end{equation}
with $\bigvee(S \cup T) \leq g$ the monomial $\tau_T\sigma_S \sigma_g^b$ has $m(T,S,b)$ strictly bigger than any other monomial in the expansion of eq.\ \eqref{eq:inutile}.
Moreover $m$ is multiplicative.

First notice that:
\begin{align*}
    d(\tau_T \sigma_S^b) &= \sum_{g \in T} (-1)^{\lvert T_{\prec g} \rvert} \tau_{T\setminus{g}} \sigma_S^b  \sigma_g \\
    &= \sum_{\substack{g \in T \\ g\textnormal{ non-critical}}} (-1)^{\lvert T_{ \prec g} \rvert} \tau_{T\setminus{g}} \sigma_S^b \sigma_g + \sum_{\substack{g \in T \\ g\textnormal{ critical}}} (-1)^{\lvert T_{ \prec g} \rvert} \tau_{T\setminus{g}} \sigma_S^b \sigma_g \\
    &= \sum_{\substack{g \in T \\ g\textnormal{ non-critical}}} (-1)^{\lvert T_{ \prec g} \rvert} \tau_{T\setminus{g}} \sigma_S^b \sigma_g + \sum_{\mathclap{\substack{ \textnormal{some } T', S', b'\\ m(T',S',b') \prec m(T,S,b)}}} \alpha_{T',S',b'} \tau_{T'} \sigma_{S'}^{b'},
\end{align*}
where $\alpha_{T',S',b'}$ are some coefficients.
In the last equality we used the relations of Theorem \ref{thm:new_gen} in order to write the non-standard monomials $\tau_{T\setminus{g}} \sigma_S^b \sigma_g$ as linear combination of standard ones.
Notice also that if the pair $(\tau_T \sigma_S^b, \tau_{T\setminus{g}} \sigma_{S\cup {g}}^{b'})$ is in $\mathcal{M}$ then $m(T,S,b)=m(T\setminus{g},S \cup \{g\}, b')$.
Hence the function $m$ is weakly decreasing on every direct path in $G_{\mathcal{M}}$, so it is constant on every directed cycle.

It is enough to prove that there are no alternating directed cycles, i.e.\ cycles such that for every pair of consecutive edges exactly one is in $\mathcal{M}$.
Suppose that there exists a directed cycle and consider two consecutive edges.
We can assume that the first one is in $\mathcal{M}$ and the second one is not.
The first edge is $(\tau_T \sigma_S^b, \tau_T \tau_g \frac{\sigma_S^b}{\sigma_g})$ for some non-critical monomial $\tau_T \sigma_S^b$ with $g$ the smaller non-critical element and $g\in S \setminus T$.
The second edge is $(\tau_T \tau_g \frac{\sigma_S^b}{\sigma_g}, \tau_{T'} \sigma_{S'}^{b'})$ for some standard monomial $\tau_{T'} \sigma_{S'}^{b'}$.
Since the value of $m$ is constant on the cycle we have that $\tau_{T'} \sigma_{S'}^{b'} = \tau_{T \setminus \{f\}} \tau_g \frac{\sigma_S^b}{\sigma_g}\sigma_f$ for some $f\in T$ non-critical for the monomial $\tau_T \tau_g \frac{\sigma_S^b}{\sigma_g}$.
These two edges are shown below.
\begin{center}
\begin{tikzpicture}[scale=.8]
  \node (a) at (0,0) {$\tau_T \sigma_S^b$};
  \node (b) at (2,2) {$\tau_T \tau_g \frac{\sigma_S^b}{\sigma_g}$};
  \node (c) at (4,0) {$\tau_{T \setminus \{f\}} \tau_g \frac{\sigma_S^b}{\sigma_g}\sigma_f$};
  \draw[->] (a) -- (b);
  \draw[->] (b) -- (c);
\end{tikzpicture}
\end{center}
The sets of critical elements for $\tau_T \sigma_S^b$ and for $\tau_T \tau_g \frac{\sigma_S^b}{\sigma_g}$ coincide, so both $g$ and $f$ are non-critical for $\tau_T \tau_g \frac{\sigma_S^b}{\sigma_g}$.
By minimality of $g$ we have $g \prec f$ and $T \prec (T \setminus \{f\}) \cup \{g\} =T'$.

We have proved that in every alternating path after two steps the set indexing the variable $\tau$ strictly increases.
Therefore there are no alternating cycles.  
\end{proof}

\subsection{Multiplicative structure}

We want to describe the product of two critical monomials in $B(P,\mathcal{G})$.

Let $(g_1, g_2, \dots, g_k)$ be a list of elements in $\mG$ and recall that $\prec$ is a linear extension of the order on $\mathcal{G}$.
Define
\[\tilde{\lambda}(g_1, g_2, \dots, g_k) = (f_1, f_2, \dots, f_k) \]
where $f_i$ is the unique $\mG$-factor of $g_1 \vee g_2 \vee \dots \vee g_i$ bigger than $g_i$ guaranteed by Proposition \ref{prop:comb_nested_sets}(1).
Define $\lambda(g_1, g_2, \dots, g_k)=\tilde{\lambda}(g_1, g_2, \dots, g_k)$ if $(f_1, f_2, \dots, f_k)$ form a $\mG$-nested set and $f_i \prec f_{i+1}$ for $i=1, \dots, k-1$. 
Set $\lambda(g_1, g_2, \dots, g_k)= 0$ otherwise.
We will use the convention that $c\mu(0)=0$ and $c\mu(\emptyset)=1$.
Let $\pi \in \mathfrak{S}_k$ be a permutation, we write $\pi(g_1, g_2, \dots, g_k)$ for the list $(g_{\pi(1)}, g_{\pi(2)}, \dots, g_{\pi(k)})$ and we denote the concatenation of two lists $T_1$ and $T_2$ by $T_1 \cup T_2$.

\begin{remark}
If $\tilde{\lambda}(g_1, g_2, \dots, g_k)=(f_1, f_2, \dots, f_k)$, then $g_1 \vee g_2 \vee \dots \vee g_j= f_1 \vee f_2 \vee \dots \vee f_j$.
Moreover, $\lambda(g_1, g_2, \dots, g_k)=0$ if there exist $i<j$ such that $g_j \leq g_i$.
Indeed, $f_j \leq f_1 \vee f_2 \vee \dots \vee f_{j-1}$ and $\{f_1, \dots, f_j\}$ is $\mG$-nested, hence $f_j = f_c$ for some $c<j$ contradicting $f_c \prec f_j$.

Let $T_1$ and $T_2$ be two lists of elements in $\mG$ and $\pi \in \mathfrak{S}_{\lvert T_1 \cup T_2 \rvert}$ be a permutation.
If $\lambda(\pi(T_1 \cup T_2))\neq 0$ then the last element of $\pi(T_1 \cup T_2)$ belongs to $\max(T_1 \cup T_2)$.

In the particular case when $\mG$ is the maximal building set and $T_1, T_2$ are chains in $\mG$, $\lambda(\pi(T_1 \cup T_2))$ is zero if $\pi$ is not a $(\lvert T_1 \rvert, \lvert T_2 \rvert)$-shuffles.
\end{remark}

The following proposition describes the multiplication of critical monomials using shuffles.
\begin{proposition}\label{thm:mult_struct}
Let $T_1$ and $T_2$ be $\mG$-nested sets. If $\cd(\bigvee (T_1 \cup T_2))< \cd(\bigvee T_1)+\cd(\bigvee T_2)$, then $c\mu(T_1)c\mu(T_2)=0$.
Otherwise
\[ c\mu(T_1)c\mu(T_2)= \sum_{\pi \in \mathfrak{S}_{\lvert T_1 \cup T_2 \rvert}} \sgn (\pi) c\mu(\lambda \pi (T_1 \cup T_2)). \let\qedsymbol\openbox\qedhere\]
\end{proposition}

Before the proof of Proposition \ref{thm:mult_struct} we need two technical lemmas.

\begin{lemma} \label{lemma:disjoint_nested}
Let $T_1$ and $T_2$ be nested sets such that $T_1 \cup T_2$ is $\mG$-nested and $\cd(\bigvee T_1)+\cd(\bigvee T_2)= \cd( \bigvee (T_1 \cup T_2))$.
Then
\[ c\mu(T_1)c\mu(T_2)=(-1)^s c\mu(T_1 \cup T_2),\]
where $s$ is the length of the permutation that reorder $T_1$ and $T_2$.
Moreover:
\[ c\mu(T_1)c\mu(T_2)= \sum_{\pi \in \mathfrak{S}_{\lvert T_1 \cup T_2 \rvert}} \sgn (\pi) c\mu(\lambda \pi (T_1 \cup T_2)). \let\qedsymbol\openbox\qedhere\]
\end{lemma}
\begin{proof}
Notice that $[\hat{0}, \bigvee (T_1 \cup T_2)]= [\hat{0}, \bigvee T_1 ] \times [\hat{0}, \bigvee T_2]$ with the same codimension, therefore
\[ c\mu(T_1)c\mu(T_2)=(-1)^s c\mu(T_1 \cup T_2).\]
Since each subset of $T_1 \cup T_2$ is $\mG$-nested, for each $\pi \in \mathfrak{S}_{\lvert T_1 \cup T_2 \rvert}$ we have $\tilde{\lambda}\pi (T_1\cup T_2)= \pi (T_1\cup T_2)$ by (3) of Proposition \ref{prop:comb_nested_sets}.
Hence $\lambda \pi (T_1 \cup T_2)$ is zero for all permutations $\pi$ except for the unique permutation that reorders $T_1$ and $T_2$.
\end{proof}

\begin{lemma}\label{lemma:ind_step_mult}
Suppose that $T$ is a $\mG$-nested set and $g\in \mG$ such that
$\cd(g \vee \bigvee T)-\cd(\bigvee T)=\cd(g)-\cd(\bigvee T_{<g})$.
Set $b= \cd(g)-\cd(\bigvee T_{<g})-1$, then
\begin{equation} \label{eq:base_step}
    c\mu(T)\tau_g \sigma_g^b= \sum_{\pi \in \mathfrak{S}_{\lvert T \rvert +1}} \sgn (\pi) c\mu(\lambda \pi (T \cup \{g\})),
\end{equation}
where the sum is taken over all permutations of $T\cup \{g\}$. 
\end{lemma}
\begin{proof}
We prove the statement by induction on $\lvert T \rvert$.

If $g \geq \bigvee T$, it follows directly from the definitions $c\mu(T)\tau_g \sigma_g^b= c\mu(T \cup \{g\})= \sum_{\pi \in \mathfrak{S}_{\lvert T \rvert +1}} \sgn (\pi) c\mu(\lambda \pi (T \cup \{g\}))$.

Let $f=g \vee \bigvee T$, $h$ be the unique $\mathcal{G}$-factor of $f$ bigger than $g$ and set $T'=T_{\leq h} \cup \{g\}$, $T''= T_{\not \leq h}$.
Notice that $T''= \emptyset$ if and only if $f \in \mG$.
Define $t'$ and $t''$ the cardinality of $T'$ and $T''$ respectively.

If $T''\neq \emptyset$, by using the inductive hypothesis and Lemma \ref{lemma:disjoint_nested}, we have
\begin{align*}
    c\mu(T)\tau_g \sigma_g^b &= (-1)^s c\mu(T'')c\mu(T' \setminus \{g\}) \tau_g \sigma_g^b \\
    &= (-1)^s c\mu(T'') \sum_{\alpha \in \mathfrak{S}_{t''}} \sgn (\alpha) c\mu(\lambda \alpha (T')) \\
    &= \sum_{\alpha \in \mathfrak{S}_{t'}} (-1)^{s+s_\alpha} \sgn (\alpha) c\mu(T'' \sqcup \lambda \alpha (T')) \\
    &= \sum_{\pi \in \mathfrak{S}_{t'+t''}} \sgn (\pi) c\mu(\lambda \pi (T \cup \{g\})),
\end{align*}
where $s$ corresponds to the permutation that reorders $T''$ and $T'\setminus \{g\}$, $s_\alpha$ to the permutation that reorders $T''$ and $\alpha(T')$.
The set $T'' \cup \lambda \alpha(T')$ is $\mG$-nested for all $\alpha \in \mathfrak{S}_{t'}$ such that $\lambda \alpha(T') \neq 0$ and in this case $\cd(\bigvee T'')+\cd(\bigvee \lambda \alpha(T'))= \cd (\bigvee (T'' \cup \lambda \alpha(T')))$.

Otherwise $T''= \emptyset$ and $f \in \mG$.
Let $Y=\{g_1, g_2, \dots, g_k\} = \max ( T \cup \{g\})$ numbered such that $g=g_k$.
We assume $g \succ t$ for all $t \in T$, the general case differs only by a sign.
We have $\prod_{i=1}^k (\tau_{g_i}-\tau_f)=0$ and so
\[ \tau_Y = \sum_{i=1}^k (-1)^{k-i} \tau_{Y \setminus \{g_i\} \cup \{f\}}. \]
Set $b(g_i)=\cd(g_i)- \cd(\bigvee T_{<g_i})-1$, for all $i\leq k$ we have $ (\sigma_{g_i} - \sigma_{f}) \prod_{j \neq i} (\tau_{g_j} -\tau_f)=0$ and
\[ 0= (\sigma_{g_i}^{b(g_i)} - \sigma_{f}^{b(g_i)}) \tau_f \prod_{j \neq i} (\tau_{g_j} -\tau_f)= (\sigma_{g_i}^{b(g_i)} - \sigma_{f}^{b(g_i)}) \tau_f \prod_{j \neq i} \tau_{g_j}, \]
so $\tau_{Y \setminus \{g_i\} \cup \{f\}} \sigma_{g_i}^{b(g_i)} = \tau_{Y \setminus \{g_i\} \cup \{f\}} \sigma_{f}^{b(g_i)}$.
Therefore we have
\begin{align*}
    c\mu(T)\tau_g \sigma_g^b &= (-1)^{s} c\mu(T\setminus Y) \prod_{i=1}^k \tau_{g_i} \sigma_{g_i}^{b(g_i)} \\
    &= (-1)^{s} c\mu(T\setminus Y) \sum_{i=1}^k (-1)^{k-i} \prod_{j \neq i} \tau_{g_j} \sigma_{g_j}^{b(g_j)} \tau_f \sigma_f^{b(g_i)} \\
    &= \sum_{i=1}^{k-1} (-1)^{t_i+1} c\mu(T\setminus \{g_i\}) \tau_g \sigma_g^b \tau_f \sigma_f^{b(g_i)} + c\mu(T\cup \{f\}),
\end{align*}
where $s$ (and $t_i$) is the length of the permutation that reorder $T\setminus Y$ and $Y\setminus \{g\}$ (respectively $T\setminus \{g_i\}$ and $\{g_i\}$).
The last summand corresponds to the identity permutation.
Apply the inductive hypothesis on the terms $c\mu(T\setminus \{g_i\}) \tau_g \sigma_g^b$ so that 
\[ (-1)^{t_i+1} c\mu(T\setminus \{g_i\}) \tau_g \sigma_g^b \tau_f \sigma_f^{b(g_i)} = \sum_{\pi} \sgn (\pi) c\mu(\lambda \pi (T \cup \{g\}))\]
where the sum is taken over all permutations $\pi$ in $\mathfrak{S}_{\lvert T \rvert +1}$ that sends the element $g_i$ in the last position.
Since every $\pi$ such that $\lambda \pi (T\cup \{g\})\neq 0$ has in the last position an element of $\max(T \cup \{g\})$, the result follows.
\end{proof}

\begin{proof}[Proof of Proposition \ref{thm:mult_struct}]
For the first part notice that $c\mu(T_i)$ is in bidegree $(2(\cd(\bigvee T_i)-\lvert T_i \lvert),\lvert T_i \rvert)$ for $i=1,2$.
Let $f=\bigvee (T_1 \cup T_2)$, the product $c\mu(T_1)c\mu(T_2)$ can be rewritten as sum of standard monomials using only relations of type
\[ \prod_{s \in S} (\tau_s- \tau_g) \prod_{t \in T} (\sigma_t - \sigma_g) \sigma_g^b \]
for $\bigvee (S\cup T)\leq g \leq f$.
The standard monomials $\tau_S \sigma_T^b$ with $\bigvee (S\cup T) \leq f$ have bidegree at most $(2(\cd(f)-\lvert S \rvert), \lvert S \rvert )$.
Therefore, if $\cd(f)< \cd(\bigvee T_1)+\cd(\bigvee T_2)$ then $c\mu(T_1)c\mu(T_2)=0$ by degree argument.

We prove the second statement by induction on $\lvert T_2 \rvert$.
The base case $T_2 = \emptyset$ is trivial.
If $\bigvee T_2 \notin \mG$ then there exist $T_3$ and $T_4$ nonempty $\mG$-nested sets such that $T_2=T_3 \sqcup T_4$ and $[\hat{0},\bigvee T_2] = [\hat{0}, \bigvee T_3]\times [\hat{0}, \bigvee T_4]$.
Applying Lemma \ref{lemma:disjoint_nested} and the inductive step we have
\begin{align*}
    c\mu(T_1) &c\mu(T_2) = (-1)^s c\mu(T_1)c\mu(T_3)  c\mu(T_4)\\
    &= (-1)^s \sum_{\alpha \in \mathfrak{S}_{t_1+t_3}} \sgn (\alpha) c\mu(\lambda \alpha (T_1 \cup T_3)) c\mu(T_4) \\
    &= (-1)^s \sum_{\alpha \in \mathfrak{S}_{t_1+t_3}} \sgn (\alpha) \sum_{\beta \in \mathfrak{S}_{t_1+t_2}} \sgn (\beta) c\mu(\lambda \beta (\alpha (T_1 \cup T_3), T_4)) \\
    &= \sum_{\pi \in \mathfrak{S}_{t_1+t_2}} \sgn (\pi) c\mu(\lambda \pi (T_1 \cup T_2)),
\end{align*}
where $t_i=\lvert T_i \rvert$.

Now we deal with the case $\bigvee T_2 \in \mG$.
Let $g=\max T_2 \in \mG$, $T'_2= T_2 \setminus \{g\}$, and $m= \cd(g)- \cd(T'_2)-1$.
We have
\begin{align*}
    c\mu(T_1) & c\mu(T_2)= c\mu(T_1)c\mu(T'_2) \tau_g \sigma_g^m  \\
    &= \sum_{\alpha \in \mathfrak{S}_{t_1+t_2-1}} \sgn (\alpha) c\mu(\lambda \alpha (T_1 \cup T'_2)) \tau_g \sigma_g^m \\
    &= \sum_{\alpha \in \mathfrak{S}_{t_1+t_2-1}} \sgn (\alpha) \sum_{\beta \in \mathfrak{S}_{t_1+t_2}} \sgn (\beta) c\mu(\lambda \beta (\alpha (T_1 \cup T'_2) \cup \{g\})) \\
    &=\sum_{\pi \in \mathfrak{S}_{t_1+t_2}} \sgn (\pi) c\mu(\lambda \pi (T_1 \cup T_2)),
\end{align*}
where we used the inductive hypothesis on $T_1$ and $T'_2$ and Lemma \ref{lemma:ind_step_mult} on $\lambda \alpha(T_1 \cup T'_2)$ and $\{g\}$.
\end{proof}

We define the algebra of critical monomials abstractly, by generators and relations.

\begin{definition} \label{def:CM}
Let $\CM(P,\mG)$ be the $\Q$-vector space generated by all the $\mG$-nested sets $T \in n(P, \mG)$ with bidegree $(2(\cd(\bigvee T)-\lvert T \rvert), \lvert T \rvert)$.
The differential is defined on the base by 
\[ \dd(T) = \sum_{t \in T \setminus \max(T)} (-1)^{\lvert T_{\prec t} \rvert} (T\setminus \{t\}) \]
and the product by $T \cdot S = 0 $ if $\cd(\bigvee (T \cup S)) < \cd(\bigvee T)+ \cd (\bigvee S)$ and 
\[ T \cdot S = \sum_{\pi \in \mathfrak{S}_{\lvert T \rvert + \lvert S \rvert}} \sgn(\pi) \lambda(\pi(T \cup S))\]
otherwise.
This structure makes $\CM(P,\mG)$ a differential bigraded algebra.
\end{definition}

We summarize all the previous results of this section in the following theorem.
\begin{theorem} \label{thm:main_CM}
The morphism $\xi \colon \CM(P,\mG) \to B(P,\mG)$ defined by $\xi(T)= c\mu(T)$ is an inclusion of differential algebras and a quasi-isomorphism.
\end{theorem}
\begin{proof}
The map $\xi$ is well defined as a morphism of $\Q$-vector spaces.
It is an inclusion since the monomials $c\mu(T)$ for $T\in n(P,\mG)$ are standard monomials and are linearly independent by Corollary \ref{cor:base_additiva_sigma_tau}.
The equality $\dd \xi = \xi \dd$ follows from Lemma \ref{lemma:diff_critical} and the equality $\xi(S \cdot T)=\xi(S)\xi(T)$ from Proposition \ref{thm:mult_struct}.
This also proves that $\CM(P,\mG)$ is a differential bigraded algebra.

Finally, the algebraic Morse theory applied to $B(P,\mG)$ and the matching $\mathcal{M}$ ensures that there exists a subcomplex $N_{\mathcal{M}}$ such that the projection 
\[ B(P,\mG) \twoheadrightarrow \faktor{B(P,\mG)}{N_{\mathcal{M}}}\]
is a quasi-isomorphism and the quotient is freely generated by critical monomials.
The composition of $\xi$ with the projection gives an isomorphism of chain complexes.
Therefore $\xi$ is a quasi-isomorphism.
\end{proof}
See Section \ref{sect:example} for an explicit example of the construction of the algebra of critical monomials.

Let $n((\hat{0},g), \mG)$ be the full subcomplex of $n(P,\mG)$ on the set of vertices $\{h \in \mG \mid h<g \}$.

All the homology groups are taken with rational coefficients.
We use the standard convention for the reduced homology that
$\tilde{H}_{-1}(\emptyset)=\Q$.

This final theorem provides an explicit description of the cohomology of the Leray model in term of cohomology of very small simplicial complexes.

\begin{theorem}\label{thm:main_Yuzvinsky}
Let $P$ be a polymatroid and $\mG$ be a building set. Then
\[ H^{\bigcdot}(B(P,\mG),\dd) \cong H^{\bigcdot}(\CM(P,\mG),\dd) \cong \bigoplus_{f \in L} \bigotimes_{g \in F} \tilde{H}_{2\cd(g)-2- \bigcdot} \Bigl(  n((\hat{0},g), \mG) \Bigr),\]
where $F=F(P,\mG,f)$ is the set of $\mG$-factors of $f$.

In particular the summand $\tilde{H}_i (n((\hat{0},g), \mG) )$ contributes in bidegree $(2 (\cd(g)-2-i), 2 +i)$.
\end{theorem}
\begin{proof}
Theorem \ref{thm:main_CM} implies
\[H(B(P,\mG),\dd) \cong H(\CM(P,\mG),\dd). \]
For each flat $f$ let $\CM_f$ be the subcomplex of $\CM(P,\mG)$ generated by all nested sets $T$ such that $\max(T)=F(P,\mG,f)$.
Moreover for each $g\in \mG$ set $\CM(g)$ to be the subcomplex of $\CM(P,\mG)$ generated by all nested sets $T$ such that $\{g\}=\max (T)$.
We have
\[ \CM(P,\mG) = \bigoplus_{f \in L} \CM_f \]
and
\[ \CM_f = \bigotimes_{g \in F(P,\mG,f)} \CM(g)\]
as complexes.
It is enough to prove that 
\[H^\bigcdot (\CM(g),\dd)=\tilde{H}_{2\cd(g)-2- \bigcdot} \Bigl(  n((\hat{0},g), \mG) \Bigr).\]
Indeed $\CM(g)$ coincides with the reduced simplicial chain complex for $n((\hat{0},g), \mG)$, under the correspondence $T \mapsto T \setminus \{g\}$.
Notice that the bidegree of $T\in \CM(g)$ is $(2(\cd(g)-\lvert T \rvert), \lvert T \rvert)$ and the degree of $T\setminus \{g\}$ in the reduced chain complex is $\lvert T \setminus \{g\} \rvert -1 = \lvert T \rvert -2$.
\end{proof}

Definition \ref{def:CM} has a straightforward generalization to integer coefficients; we left open the following question.
\begin{conjecture}
Does Theorem \ref{thm:main_Yuzvinsky} generalizes to integer coefficients?
\end{conjecture}
The analogous statement in the realizable case with maximal building set was proven in \cite{DGMP,dLS}.

\section{K\"ahler package} \label{sect:Kahler}

Let $\DP^\bigcdot (P,\mG)$ be the graded algebra $B^{2 \bigcdot,0}(P,\mG)$.
This algebra, in the realizable case, is the Chow ring of the De Concini Procesi wonderful model for the subspace arrangement.
A presentation of $\DP(P,\mG)$ is given by the generators $x_g$ for $g \in \mG$ with relations 
\[ x_S c_g^b \]
where $S\subseteq \mG$, $g \in \mG$ and $b \geq \cd(g)- \cd (\bigvee S_{<g})$.
The algebra $\DP(P,\mG)$ has an additive basis given by
    \[x_S^b\]
where $S \in n(P, \mathcal{G})$
and for each $s\in S$ we have that $0<b(s)< \cd(s) - \cd (\bigvee (S)_{<s})$, see Corollary \ref{cor:base_additiva_e_x}.

A second presentation is given by the generators $\sigma_g$ for $g \in \mG$ with relations
\[ \sigma_g^b  \prod_{s \in S} (\sigma_s-\sigma_g) \]
where $\bigvee S \leq g$ and $b = \cd(g)- \cd (\bigvee S)$, see Theorem \ref{thm:new_gen}. 
The algebra $\DP(P,\mG)$ has an additive basis given by
    \[\sigma_S^b\]
where $S \in n(P, \mathcal{G})$
and for each $s\in S$ we have that $0<b(s)< \cd(s) - \cd (\bigvee S_{<s})$, see Corollary \ref{cor:base_additiva_sigma_tau}.

\begin{remark}\label{rem:top_in_G}
If $\hat{1} \notin \mG$ then the polymatroid $P$ is direct sum of other polymatroids.
Indeed, let $a_1, \dots, a_k$ be the $\mG$-factors of $\hat{1}$, the poset $L$ is a product $\prod_{i=1}^k [\hat{0}, a_k]$.
There exist polymatroids $P^{a_i}$ (defined in the following, see Lemma \ref{lemma:restr_contr}) such that $P=\oplus_{i=1}^k P^{a_i}$ and building sets $\mG^{a_i}= \mG \cap [\hat{0},a_i]$.
Moreover, $\DP(P,\mG)= \otimes_{i=1}^k \DP(P^{a_i}, \mG^{a_i})$ and the dimension of $\DP(P,\mG)$ is $\cd(\hat{1})- \lvert F(P,\mG,\hat{1}) \rvert$ (where $k= \lvert F(P,\mG,\hat{1}) \rvert $).
\end{remark}

For the clarity of exposition, we assume $\hat{1} \in \mG$ in this section.
Consider the isomorphism $\deg \colon \DP^{\cd(\hat{1})-1} (P,\mG) \to \Q$  defined by 
\[\deg (x_{\hat{1}}^{\cd(\hat{1})-1})=(-1)^{\cd(\hat{1})-1}.\]

\begin{definition}
Let $A$ be a graded algebra with top degree $n$ and $\deg \colon A^n \to \Q$ an isomorphism.
We say that
\begin{itemize}
    \item the algebra $A$ satisfies \textit{Poincaré duality} if the bilinear pairing
    \[ A^k \times A^{n-k} \to \Q\]
    defined by $(a,b) \mapsto \deg(ab)$ is non-degenerate.
    \item the element $\ell \in A^1$ satisfies the \textit{Hard Lefschetz property} if the multiplication map
    \[ \cdot \ell^{n-2k} \colon A^k \to A^{n-k}\]
    is an isomorphism for all $k\leq \frac{n}{2}$.
    \item the element $\ell \in A^1$ satisfies the \textit{Hodge-Riemann relations} if the bilinear form 
    \[ Q_{\ell}^k \colon A^k \times A^k \to \Q \]
    defined by $Q_{\ell}^k(a,b)=(-1)^k \deg(a \ell^{n-2k}b)$ (for $k\leq \frac{n}{2}$) is positive definite on the subspace
    \[ P_k = \ker (\cdot \ell^{n-2k+1} \colon A^k \to A^{n-k+1}).\]
\end{itemize}
We will abbreviate these properties with $\PD_A$, $\HL_A(\ell)$, and $\HR_A(\ell)$ respectively. 
\end{definition}

\subsection{Poincaré duality}
In this subsection we give a direct proof of the Poincaré duality property for $\DP(P,\mG)$.

\begin{definition}\label{def:dual_monomial}
Suppose that $\hat{1}\in \mG$ and let $x_S^b$ be a standard monomial.
The element $\epsilon(x_S^b)$ is 
\[ \epsilon(x_S^b) = (-1)^{\lvert S\setminus \{ \hat{1} \} \rvert} x_{S^+}^c, \]
where $S^+ = S \cup \{ \hat{1} \}$, $c(\hat{1})= \cd(\hat{1})- \cd (\bigvee S_{<\hat{1}})-b(\hat{1})-1$, and  $c(g)= \cd(g)- \cd (\bigvee S_{<g})-b(g)$ for $g\in S \setminus \{\hat{1}\}$.

We will write $c_S$ instead of $c$ when we want to stress the dependency on $S$ and $b$.
\end{definition}

Recall the chosen monomial order with the property that if $h>g$ then $h\succ g$ and $x_h \prec x_g$.
We fix the basis of $\DP^k$ consisting in all standard monomials $x_S^b$ of degree $k$ ordered with the aforementioned monomial order.
In complementary degree $\DP^{\cd(\hat{1})-k}$, we consider the basis given by $\epsilon(x_S^b)$ ordered using the monomial order on $x_S^b$.
In order to prove Poincaré duality we will show that the matrix with entries $\deg(x_S^b \epsilon(x_T^c))$ is non-degenerate.
Lemma \ref{lemma:poinc_matr_ones_diag} proves that the matrix has values $\pm 1$ on the diagonal and Lemma \ref{lemma:poin_matr_upper_triang} shows that the matrix is upper triangular.

\begin{lemma}\label{lemma:poinc_matr_ones_diag}
If $\hat{1}\in \mG$ then for all standard monomials we have
\[ x_S^b \epsilon(x_S^b)= x_{\hat{1}}^{\cd(\hat{1})-1}.\]
\end{lemma}
\begin{proof}
We prove the statement by induction on $\lvert S \setminus \{ \hat{1} \} \rvert$.
The base case $S=\{\hat{1}\}$ is trivial.
For the inductive step we choose $g \in \max (S_{<\hat{1}})$ and set $T=S\setminus \{g,\hat{1}\}$.
For the sake of notation, let $n(h)=b(h)+c_S(h)$ for all $h\in S^+$ (where $c_S(h)$ is introduced in Definition \ref{def:dual_monomial}).
Notice that $x_T x_f x_{\hat{1}}^{n(\hat{1})}=0$ for all $f\in (g,\hat{1}) \cap \mG$, because $f \vee \bigvee T > g \vee \bigvee T$.
Since $x_T^n\sigma_g^{n(g)}=0$ by relation (ii), we have
\begin{align*}
    0 &= x_T^n\sigma_g^{n(g)}x_{\hat{1}}^{n(\hat{1})} \\
    &= x_T^n (x_g+x_{\hat{1}})^{n(g)}x_{\hat{1}}^{n(\hat{1})} \\
    &= x_T^n (x_g^{n(g)} +x_{\hat{1}}^{n(g)}) x_{\hat{1}}^{n(\hat{1})},
\end{align*}
where in the last equality we used $x_T x_g x_{\hat{1}}^{n(\hat{1})+1}=0$.
Therefore,
\begin{align*}
    x_S^b \epsilon(x_S^b) &= (-1)^{\lvert S\setminus \{ \hat{1} \} \rvert} x_T^n x_g^{n(g)} x_{\hat{1}}^{n(\hat{1})} \\
    &= (-1)^{\lvert S\setminus \{ \hat{1} \} \rvert -1} x_T^n x_{\hat{1}} ^{n(g)} x_{\hat{1}}^{n(\hat{1})}\\
    &= (-1)^{\lvert T\setminus \{ \hat{1} \} \rvert } x_T^n x_{\hat{1}} ^{n(g)+n(\hat{1})} \\
    &= x_T^b \epsilon(x_T^b) = x_{\hat{1}}^{\cd(\hat{1})-1},
\end{align*}
by inductive hypotheses on $T$.
\end{proof}

Let $d_S$ be the function defined by $d_S(\hat{1})= \cd(\hat{1}) -\cd(\bigvee S_{< \hat{1}})-1$ and by $d_S(g)= \cd(g) -\cd(\bigvee S_{< g})$ for $g \neq \hat{1}$.
\begin{lemma} \label{lemma:x_S_b_is_zero}
Let $S$ be a nested set, $g \in S$ and $x_S^b$ be a monomial such that for all $h>g$ we have $b(h)\geq d_S(h)$ and $b(g) > d_S(h)$. Then $x_S^b=0$.
\end{lemma}
The proof of the lemma is the same of \cite[Lemma 5.4.1 (b)]{BDF}.
Recall the chosen monomial order with the property that if $h>g$ then $h\succ g$ and $x_h \prec x_g$.
We need the following statement.
\begin{lemma} \label{lemma:poin_matr_upper_triang}
Let $x_S^b$ and $x_T^c$ be two standard monomials in $\DP^k(P,\mG)$ such that $x_S^b \prec_{revlex} x_T^c$. Then $x_S^b\epsilon(x_T^c)=0$.
\end{lemma}
\begin{proof}
Consider $T'$ and $c'$ such that $x_{T'}^{c'}=\epsilon(x_T^c)$ and notice that $T' \setminus \{ \hat{1}\}= T \setminus \{ \hat{1}\}$.
Define $g=\max_{\prec} \{h \mid b(h) \neq c(h)\}$ and, by hypothesis, $b(g)>c(g)$.
If $S\cup T'$ is not $\mG$-nested then we have $x_S^b\epsilon(x_T^c)=0$.
Otherwise set $A= (S\cup T')_{\geq g}$, by (4) of Proposition \ref{prop:comb_nested_sets} we have that $A$ is a chain $(a_1 < a_2 < \dots < a_l) $ with $a_1=g$.
For $a_i \neq g, \hat{1}$ we have
\begin{align*}
    b(a_i)+c'(a_i) &= b(a_i)+ \cd(a_i) - \cd( \bigvee T_{<a_i})-c(a_i) \\
    &= \cd(a_i) - \cd( \bigvee T'_{<a_i}) \\
    &\geq \cd(a_i) - \cd( \bigvee (S\cup T')_{<a_i}) = d_{S\cup T'} (a_i).
\end{align*}
The same holds for $\hat{1}$ (the proof has a minus one in the mid steps). For $a_1$ we have $b(g)+c'(g)> d_{S\cup T'} (g)$ because $b(g)>c(g)$.
Therefore the monomial $x_S^b\epsilon(x_T^c)=x_{S\cup T'}^{b+c'}$ satisfies the hypothesis of Lemma \ref{lemma:x_S_b_is_zero} and we obtain the claimed result $x_S^b\epsilon(x_T^c)=0$.
\end{proof}

Finally we can prove the Poincaré duality property:
\begin{theorem}[Poincaré duality] \label{thm:Poinc_duality}
If $\hat{1}\in \mG$ then the algebra $\DP(P,\mG)$ is a Poincaré duality algebra of dimension $\cd(\hat{1})-1$.

More generally, $\DP(P,\mG)$ is a Poincaré duality algebra of dimension $\cd(\hat{1})-\lvert F(P,\mG,\hat{1}) \rvert$.
\end{theorem}
\begin{proof}
The function $\epsilon$ has the property $\epsilon^2= \id$, and gives a bijection between standard monomials in degree $k$ and in degree $r-k$.
This, together with Corollary \ref{cor:base_additiva_e_x}, shows that $\dim \DP^k(P,\mG)= \dim \DP^{r-k}(P,\mG)$.
We consider on standard monomials the reverse lexicographical order.
Lemma \ref{lemma:poin_matr_upper_triang} ensures that the matrix of the Poincaré pairing (in the chosen basis) is upper triangular.
From Lemma \ref{lemma:poinc_matr_ones_diag} we obtain that the entries on the diagonal are $\pm 1$ and so the Poincaré pairing is non degenerate.
The last statement follows from the first one together with Remark \ref{rem:top_in_G}.
\end{proof}

We remark that the bases of standard monomials $\{x_S^b \}$ and $\{(-1)^r \epsilon (x_S^b)\}$ are not dual bases.

\subsection{Tensor decomposition} \label{subsect:tensor_dec}

This technical section is devoted to computing the annihilator $\Ann(\sigma_g)$ and $\Ann (x_g)$ for $g\in \mG$.
We describe it using the Chow ring of different polymatroids: $\tr_g P$, $P^g$ and $P_g$. In the case of matroids these operations are known as truncation, restriction, and contraction.

The following proposition is needed for the proof of the main result of this section.
\begin{proposition} \label{prop:Poinc_alg}
Let $A$ and $B$ be Poincaré duality algebra of the same dimension $n$, then:
\begin{itemize}
    \item for each $x \in A^k$, $x \neq 0$, the ring $A/\Ann(x)$ is a Poincaré duality algebra of dimension $n-k$,
    \item each surjective homomorphism $f \colon A \to B$ is an isomorphism. \let\qedsymbol\openbox\qedhere
\end{itemize}
\end{proposition}
The proof of the above proposition can be found, for example, in \cite[Proposition 7.2, Proposition 7.13]{AHK}.

Let $P=(E,\cd)$ be a polymatroid with building set $\mG$.
Consider $g \in \mG$ such that $\cd(g)>1$ and let $\tr_g \cd \colon 2^{E} \to \N$ be the function defined by:
\[ \tr_g \cd (h) = \begin{cases} \cd(h)-1 & \textnormal{if } \cd (h) = \cd( h \cup  g), \\
\cd(h) &\textnormal{otherwise.}
\end{cases}\]
We denote by $\tr_g L$ the poset of flats of $\tr_g \cd$.
Finally, define 
\[\tr_g \mG=\{\overline{h} \in \tr_g L \mid h \in \mG  \},\]
where $\overline{h}$ is the closure with respect to $\tr_g \cd$ of the flat $h$. Notice that $\tr_g L$ is a subposet of $L$ but with a different codimension function.

\begin{lemma} \label{lemma:tr_p_is_polymatroid}
For all $g\in \mG$ with $\cd(g)>1$, the pair $\tr_g P= (E,\tr_g \cd)$ is a polymatroid and $\tr_g \mG$ is a building set for the poset of flats $\tr_g L$.
\end{lemma}
\begin{proof}
It is easy to see that $(E,\tr_g \cd)$ is a polymatroid. Let $x \in \tr_g L$ and notice that, for all $h \in \mG$, $h\leq x$ in $L$ if and only if $\overline{h} \leq x$ in $\tr_g L$.
Thus, we have $\max \tr_g \mG_{\leq x}= \max \mG_{\leq x}$ and it follows that
    \[
      [\hat{0},x] \simeq \prod_{y \in \max(\mathcal{G}_{\leq x})} [\hat{0},y] \simeq \prod_{y \in \max(\tr_g \mG_{\leq x})} [\hat{0},y].
    \]
For the second part of the definition of a building set we have two cases. Let $\{y_1,\dots,y_n\}=\max \mG_{\leq x}$  and assume $g \nleq x$, which implies $g \nleq y_i$ for every $i$:
    \[\tr_g \cd (x)=\cd(x)=\sum_{y \in \max \mG_{\leq x}} \cd(y)=\sum_{y \in \max \tr_{g}\mG_{\leq x}} \tr_g\cd(y).\]
Finally, let $g \leq x$ then by Proposition \ref{prop:comb_nested_sets} there exists only one $h_i$ such that $g \leq h_i$. Thus, we have the following:
    \[\tr_g \cd (x)=\cd(x)-1= \Bigl( \sum_{y \in \max \mG_{\leq x}} \cd(y) \Bigr) -1 = \sum_{y \in \max \tr_{g}\mG_{\leq x}} \tr_g\cd(y).\]
This concludes the proof.
\end{proof}

Define the map 
\[\zeta_g \colon \DP(\tr_g P, \tr_g \mG) \to \DP(P,\mG)/\Ann(\sigma_g)\]
by $\zeta_g(\sigma_k)= \sigma_{h}$ where $h$ is any element in $\mG$ such that $\overline{h}=k$.

\begin{remark}
In the realizable case, this construction can be viewed geometrically: consider a generic hyperplane $H$ containing the flat $g$.
The intersection of the subspace arrangement with $H$ describes a subspace arrangement in $H$ whose poset of intersection is $\tr_g L$. Moreover, the natural closed inclusion between the two wonderful compactification induces a surjective map $\DP(P,\mG) \to \DP(\tr_g P,\tr_g \mG) $ with kernel $\Ann(\sigma_g)$. The map $\zeta_g$ is its pseudo-inverse.
\end{remark}

\begin{lemma}\label{lemma:zeta_g}
For $g \in \mG$ with $\cd(g)>1$, the map $\zeta_g$ is well defined and an isomorphism.
Moreover $\deg (\alpha)= \deg (-\sigma_g \zeta_g (\alpha))$ for all $\alpha\in \DP(\tr_g P, \tr_g \mG)$.
\end{lemma}
\begin{proof}
We show that the map $\zeta_g$ does not depend on the choice of $h$: suppose that exist $h,f \in \mG$ such that $\overline{h}=\overline{f}$. By symmetry we may assume $h \not \geq f$.
Since $g\vee h =\overline{h}=\overline{f} =g \vee f$, we have $\overline{h}\in \mG$, so replacing $f$ with $g\vee f$ we assume $f>h$.
Notice that $\cd(f)=\cd(h)+1$ and $f=g \vee h$ so
\[
\sigma_g (\sigma_h-\sigma_f) = \sigma_f (\sigma_h-\sigma_{f}) =0.
\]

We verify that the relations (i) and (ii) of Theorem \ref{thm:new_gen} are send to zero.
Consider an antichain $A \subset \tr_g \mG$ and $k\in \tr_g \mG$ such that $k\geq \bigvee A$, set $n=\tr_g \cd (k)- \tr_g \cd (\bigvee A)$.
Let $h\in \mG$ such that $\overline{h}=k$ and $B\subset \mG$ such that $\overline{b_i}=a_i$ for all $i$.
We have
\[ \sigma_g \zeta_g \Bigl( \sigma_k^n \prod_{a \in A} (\sigma_a-\sigma_k) \Bigr) = \sigma_g \sigma_h^n \prod_{b \in B} (\sigma_b-\sigma_h).\]
Notice that $\cd(h)-\cd (\bigvee B)=n$ unless $h\geq g$ and $\bigvee B \not \geq g$ in which case $\cd(h)-\cd (\bigvee B)=n+1$.
The non trivial case is the latter. Notice also that $h=g \vee \bigvee B$. We use the relations to obtain:
\begin{align*}
    \sigma_g \zeta_g \Bigl( \sigma_k^n \prod_{a \in A} (\sigma_a-\sigma_k) \Bigr) &= \sigma_g \sigma_h^n \prod_{b \in B} (\sigma_b-\sigma_h) \\
    &= \sigma_h^{n+1} \prod_{b \in B} (\sigma_b-\sigma_h) \\
    &=0.
\end{align*}
Hence $\zeta_g$ is well defined.

The map is surjective because for each $h\in \mG$ we have $\zeta_g (\sigma_{\overline{h}})=\sigma_h$.
Finally applying Proposition \ref{prop:Poinc_alg} we obtain the sought isomorphism.

For the last statement it is enough to notice that $\sigma_g\zeta_g(x_{\hat{1}}^{r-1})=x_{\hat{1}}^{r}$.
\end{proof}

Let $P=(E,\cd)$ be a polymatroid, $\mG$ be a building set and $g \in \mG$ any element.
The restriction of the polymatroid to the flat $g$ is $P^g=(E^g, \cd^g)$ where $E^g=\{h \in E \mid h\leq g\}$ and $\cd^g=\cd_{|E^g}$.
The contraction of $P=(E,\cd)$ to the flat $g$ is $P_g=(E_g, \cd_g)$ where $E_g=E \setminus E^g$ and $\cd_g(h)= \cd(h\vee g)- \cd(g)$.

Define $L^g=[\hat{0},g]$, $\mG^g=\mG \cap L^g$, $L_g=[g,\hat{1}]$, and
\[ \mG_g =\{h\vee g \mid h \in \mG \setminus [\hat{0},g] \}.\]


The proof of the following lemma is analogous to the one of Lemma \ref{lemma:tr_p_is_polymatroid}, so we omit it.
\begin{lemma}\label{lemma:restr_contr}
The restriction and the contraction at $g\in \mG$ are polymatroids with poset of flats $L^g$ (respectively $L_g$) and building set $\mG^g$ (resp.\ $\mG_g$).
\end{lemma}

\begin{remark}
In the case of matroids $M$, we have for every $e \in E$ that $M_e=\tr_e M$ is the contraction of the matroid.
\end{remark}

Define the map 
\[ \psi_g \colon \DP(P^g, \mG^g) \otimes \DP(P_g, \mG_g) \to \faktor{\DP(P,\mG)}{\Ann (x_g)}\]
by $\psi_g(\sigma_h \otimes 1)= \sigma_h$ and $\psi_g(1 \otimes \sigma_{g \vee h})= \sigma_h$.

\begin{lemma}\label{lemma:iso_psi_g}
For all $g \in \mG \setminus \{\hat{1}\}$ the map $\psi_g$ is well defined and it is an isomorphism.
Moreover, $\deg(\alpha)\deg(\beta)=\deg(x_g\psi_g(\alpha \otimes \beta))$ for all $\alpha \in \DP(P^g, \mG^g)$ and $\beta \in \DP(P_g, \mG_g)$.
\end{lemma}
\begin{proof}
We verify that $\psi_g(1 \otimes \sigma_{g\vee h})$ does not depend on the choice of the element $h$.
Suppose that there exist $h,f \in \mG$ such that $g\vee h=g \vee f$ and $h,f \not \leq g$. By symmetry we may assume $h \not \geq f$.
Replacing $f$ with $g\vee f$ we assume $f>h$, then 
\[
x_g (\sigma_h-\sigma_f) = x_g \sum_{\substack{l\geq h \\ l \not \geq g}} x_l=0,
\]
because $\{g,l\}$ cannot be $\mG$-nested since $g<f\leq g\vee l$ and $l\not \geq g$.

We verify that all relations in the domain are mapped to zero.
The ones in $\DP(P^g,\mG^g)$ hold also in $\DP(P,\mG)$ trivially.
Consider $h\in \mG$ and $S \subset \mG$ an antichain such that $\bigvee S \leq h$ and $s \not \leq g$ for all $s \in S$. 
Set $n=\cd(g\vee h)-\cd(g \vee \bigvee S)$.
There are two cases:
\begin{itemize}
    \item if $g \vee h \not \in \mG$ then $n=\cd (h)- \cd (\bigvee S)$ and
    \[ x_g \psi_g \Bigl( 1 \otimes \sigma_{g\vee h}^n \prod_{a \in S} (\sigma_{g \vee s}-\sigma_{g \vee h}) \Bigr) = x_g  \sigma_h^n \prod_{a \in S} (\sigma_s-\sigma_h)=0,\]
    \item if $g \vee h \in \mG$ then 
    \begin{align*}
        x_g \psi_g \Bigl( 1 \otimes \sigma_{g\vee h}^n \prod_{s \in S} (\sigma_{g \vee s}-\sigma_{g \vee h}) \Bigr) &= x_g  \sigma_{g\vee h}^n \prod_{s \in S} (\sigma_s-\sigma_{g\vee h})\\
        &= \sum_{A} x_g x_A \sigma_{g\vee h}^n,
    \end{align*}
    where the sum is taken over all sets $A=\{a_1,\dots a_k\}$ such that $a_i \geq s_i$ and $a_i \not \geq g \vee h$.
    Applying Lemma \ref{lemma_tecnico} to $g\vee h$, $S \cup \{g\}$ and $A \cup \{g\}$ we obtain that each term $x_g x_A \sigma_{g\vee h}^n$ is zero.
\end{itemize}

The map $\psi_g$ is surjective because either $h \in \mG^g$ or $g \vee h \in \mG_g$ for all $h\in \mG$.
We apply Proposition \ref{prop:Poinc_alg}, $\DP(L,\mG)/\Ann(x_g)$ is a Poincaré duality algebra of dimension $\cd(\hat{1})-2$. The algebra $\DP(P^g, \mG^g) \otimes \DP(P_g, \mG_g)$ is Poincaré duality of dimension $(\cd(g)-1) + (\cd(\hat{1})-\cd(g)-1)$ (here is the only point were we use $g\neq \hat{1}$).
Since $\psi_g$ is surjective between Poincaré duality algebras of the same dimension, it is an isomorphism.

For the last statement we have 
\begin{align*}
    x_g \psi_g(x_g^{\cd(g)-1} \otimes x_{\hat{1}}^{\cd(\hat{1})-\cd(g)-1})
    &= x_g \sigma_g^{\cd(g)-1}x_{\hat{1}}^{\cd(\hat{1})-\cd(g)-1} \\
    &= (x_g-\sigma_g) \sigma_g^{\cd(g)-1} x_{\hat{1}}^{\cd(\hat{1})-\cd(g)-1} \\
    &= -x_{\hat{1}} \sigma_g^{\cd(g)-1}x_{\hat{1}}^{\cd(\hat{1})-\cd(g)-1} \\
    &= - x_{\hat{1}}^{\cd(\hat{1})-1},
\end{align*}
so $\deg(x_g^{\cd(g)-1})\deg(x_{\hat{1}}^{\cd(\hat{1})-\cd(g)-1})=(-1)^{\cd(\hat{1})}=\deg(- x_{\hat{1}}^{\cd(\hat{1})-1})$.
\end{proof}

\subsection{Hard Lefschetz and Hodge-Riemann}

We define a simplicial cone $\Sigma \subset \DP^1(P,\mG)$ and we will show that each element $\ell \in \Sigma$ satisfies Hard Lefschetz and Hodge-Riemann relations.
\begin{definition}
The $\sigma$-\textit{cone} $\Sigma_{P,\mG} \subset \DP^1(P,\mG)$ is the convex cone
\[\Sigma_{P,\mG}=\Bigl\{ - \sum_{g \in \mG} d_g \sigma_g \mid d_g> 0
\Bigr\}. \let\qedsymbol\openbox\qedhere\]
\end{definition}

Let $a\in E$ be an atom in $L$, i.e.\ the interval $(\hat{0},a)$ is empty.
Consider the set
\begin{equation} \label{eq:a_set}
    \{ g \in \mG \setminus \{a\} \mid g \neq \overline{S} \textnormal{ for all } S \subseteq E \setminus \{a\} \},
\end{equation}
of all elements $g\in \mG$ that cannot be written as the closure of some subset $S\subset E$ not containing $a$.
Define $E(a)$ as the disjoint union of $E\setminus \{a\}$ and the minimal elements of the set in \eqref{eq:a_set}.
Define the pair $P(a)=(E(a), \cd)$, where with a slight abuse of notation \[\cd(\{e_1, \dots, e_l, g_1, \dots, g_k\})= \cd(\{e_1, \dots, e_l\} \cup g_1 \cup  \dots \cup g_k\}).\]
We also define $\mG(a)=\mG \setminus \{a\}$.
The polymatroid $P(a)$ depends on $\mG$ but we omit this dependency in our notation.

In the realizable case, this polymatroidal operation corresponds to remove only the subspace $S_a$ from the building set $\mathcal{G}$ and from the arrangement $\mathcal{A}$.
Now, there are subspaces in the lattice of flats $\mathcal{L}_{\mathcal{A}}$ that are not flats of $\mathcal{A} \setminus S_a$.
Among them we want to keep trace only of the ones blown up, i.e.\ belonging to $\mG$; so we add to the deleted arrangement $\mathcal{A} \setminus S_a$ all the flats corresponding to minimal elements in the set \eqref{eq:a_set}.

\begin{lemma}
The pair $P(a)=(E(a),\cd)$ is a polymatroid and $\mG(a)$ is a building set for the poset of flats of $P(a)$.
\end{lemma}
\begin{proof}
It is easy to see that $(E(a),\cd)$ is a polymatroid and that the lattice of flats $L_{P(a)}$ of $P(a)$ is a subposet of the lattice of flats $L$ of $P$.
We verify that $\mathcal{G}(a)$ is a building set. 
We check the definition for all $x \in L_{P(a)}$: if $a$ is not a $\mG$-factor of $x$ then $\max(\mG_{\leq x})=\max(\mG(a)_{\leq x})$ and it follows from the properties of $\mG$.
Otherwise, $a$ is a $\mG$-factor of $x$ and $x$ cannot lie in the lattice $L_{P(a)}$ generated by $\mG \setminus \{a\}$.
\end{proof}

\begin{lemma}\label{lemma:iso_p_a}
For an atom $a\in E$, $a\neq \hat{1}$, consider the element $\mu_0= (x_a-\sigma_a)^{\cd(a)}$.
There exists an isomorphism:
\[p_a \colon \DP(P_a,\mG_a) \to \faktor{\DP(P(a),\mG(a))}{\Ann (\mu_0)}. \]
Moreover, $\deg(\alpha)=\deg(\mu_0 p_a(\alpha))$ for all $\alpha \in \DP(P_a,\mG_a)$.
\end{lemma}
\begin{proof}
Notice that $\mu_0= (x_a-\sigma_a)^{\cd(a)}$ is a multiple of $x_a$ because $\sigma_a^{\cd(a)}=0$, hence $\Ann(x_a) \subseteq \Ann(\mu_0)$.
Define the morphism $p_a$ as the composition 
\[\DP(P_a,\mG_a) \hookrightarrow \DP(P^a,\mG^a) \otimes \DP(P_a,\mG_a) \xrightarrow{\psi_a} \DP(P,\mG)/\Ann(x_a) \twoheadrightarrow \DP(P,\mG)/\Ann(\mu_0),\]
where the first map is the inclusion $x \mapsto 1 \otimes x$.
Explicitly $p_a(\sigma_{a\vee h})=[\sigma_h]$ for all $h \neq a$.
Since $\mG(a)$ is a subset of $\mG$, $\DP(P(a),\mG(a))$ is a subalgebra of $\DP(P,\mG)$.
The range of the map $p_a$ is equal to $\DP(P(a),\mG(a))/\Ann (\mu_0)$, so the morphism in the statement is well defined and surjective.
Since $a\neq \hat{1}$ we have $\mu_0 \neq 0$ and by Proposition \ref{prop:Poinc_alg} the map $p_a$ is an isomorphism, because both algebras satisfy Poincaré duality of dimension $\cd(\hat{1})-\cd(a)-1$.

For the last statement we have $\mu_0 p_a(x_{\hat{1}}^{\cd(\hat{1})-\cd(a)-1})=
(-1)^{\cd(a)} x_{\hat{1}}^{\cd(\hat{1})-1}$ 
and so $\deg(x_{\hat{1}}^{\cd(\hat{1})-\cd(a)-1})=(-1)^{\cd(\hat{1})-\cd(a)-1}= \deg((-1)^{\cd(a)} x_{\hat{1}}^{\cd(\hat{1})-1})$.
\end{proof}

\begin{lemma}\label{lemma:A_is_correct}
Let $a\in E$, $a\neq \hat{1}$, be an atom and $\mu_0= (x_a-\sigma_a)^{\cd(a)}$.
Consider the polynomial $p(x)=\sum_{i=0}^{\cd(a)} \binom{\cd(a)}{i} x^i (x_a-\sigma_a)^{\cd(a)-i}$, then 
\[ \faktor{\DP(P(a),\mG(a))[x]}{(x\Ann (\mu_0),p(x))} \cong \DP(P,\mG). \let\qedsymbol\openbox\qedhere\]
\end{lemma}
\begin{proof}
Define the morphism 
\[ \DP(P(a),\mG(a))[x] \to \DP(P,\mG)\]
by $\sigma_g \mapsto \sigma_g$ and $x \mapsto -x_a$.
By Lemmas \ref{lemma:iso_p_a} and \ref{lemma:iso_psi_g} the elements of the form $x\Ann(\mu_0)$ are in the kernel.
Also $p(x)$ is in the kernel because its image is $(-\sigma_a)^{\cd(a)}=0$.
Clearly, the map is surjective.

Notice that if $A$ is a Poincaré duality algebra and $p(x) \in A[x]$ a monic polynomial with constant term $\mu_0$ then $A[x]/(x \Ann(\mu_0), p(x))$ is a Poincaré duality algebra. Indeed, if a generic element $\sum_{i=0}^{j} a_i x^i$ (with $a_j \not \in \Ann(\mu_0)$ and $j< \deg (p)$) of degree $k$ is orthogonal to all elements of degree $n-k$, then $(\sum_{i=0}^{j} a_i x^i)a'=0$ for all $a' \in A^{n-k}$. This implies $a_0a'=0$ and $a_0=0$.
Moreover, $(\sum_{i=1}^{j} a_i x^i)a'x^{\deg(p)-j}=0$ implies $a_ja' \mu_0=0$ and $a_j\mu_0=0$ by Poincaré duality in $A$, contradicting the fact $a_j \not \in \Ann(\mu_0)$.
In particular, $\DP(P(a),\mG(a))[x]/(x\Ann (\mu_0),p(x))$ is a Poincaré duality algebra of dimension $\cd(\hat{1})-1$.

The map $\DP(P(a),\mG(a))[x] \to \DP(P,\mG)$ is injective by Proposition \ref{prop:Poinc_alg} because domain and codomain are Poincaré duality algebras of the same dimension equal to $\cd(\hat{1})-1$.
\end{proof}

The following theorem provides an abstract procedure to prove the Hodge-Riemann relations inductively.

\begin{theorem}\label{thm:HR_inductive_step}
Let $C$ be a Poincaré duality algebra and $p(x)=x^d+ \mu_{d-1} x^{d-1} + \dots + \mu_0=0 \in C[x]$ be a homogeneous polynomial with $\mu_0 \neq 0$. Let $B=C/\Ann(\mu_0)$ and $A=C[x]/(x\Ann(\mu_0), p(x))$.
Let $\ell \in C^1$ be an element satisfying $\HR_C(\ell)$ and $\HR_B(\ell)$.
Then $\HR_A(\ell+\epsilon x)$ holds for sufficiently small positive $\epsilon$.
\end{theorem}
In the above theorem the degree function on $B$ is induced by $\mu_0$, i.e.\ $\deg_B(\alpha)=\deg_C(\alpha \mu_0)$.
Since the top degrees coincide $A^{top}=C^{top}$, we also implicitly assume that $\deg_A=\deg_C$.

The proof of Theorem \ref{thm:HR_inductive_step} is the same of the proof of \cite[Proposition 8.2]{AHK}, so we omit it.

The following easy lemma shows that the maps introduced in Section \ref{subsect:tensor_dec} preserve the $\Sigma$-cone.

\begin{lemma} \label{lemma:cone} \leavevmode
The following holds:
\begin{enumerate} 
\item For any $g \in \mG$, $g \neq \hat{1}$ the natural map
\[ \DP^1(P,\mG) \to \DP^1(P^g,\mG^g) \oplus \DP^1(P_g,\mG_g) \]
induced by the quotient by $\Ann(x_g)$ composed with $\psi_g^{-1}$, maps $\Sigma_{P,\mG}$ into $\Sigma_{P^g,\mG^g} \times \Sigma_{P_g,\mG_g}$. 
\item \label{lemma:zeta_g_cone} For any $g \in \mG$ the morphism 
\[ \DP^1(P,\mG) \to \DP^1(\tr_{g}P,\tr_{g}\mG)\]
induced by the quotient by $\Ann(\sigma_{g})$ composed with $\zeta_{g}^{-1}$, maps $\Sigma_{P,\mG}$ into $\Sigma_{\tr_{g}P,\tr_{g}\mG}$. 
\item  \label{lemma:p_a_cone} For any atom $a \in E, a \neq \hat{1}$ the natural map 
\[\DP^{1}(P(a), \mG(a)) \to \DP^{1}(P_a,\mG_a)\]
induced by the quotient by $\Ann (\mu_0)$ composed with $p_a^{-1}$, maps $\Sigma_{P(a),\mG(a)}$ into $\Sigma_{P_a,\mG_a}$. \let\qedsymbol\openbox\qedhere
\end{enumerate}
\end{lemma}
\begin{proof} \leavevmode
\begin{enumerate}
    \item Let $l=- \sum_{h \in \mG} d_h \sigma_h$ be an element of the $\sigma$-cone, we have
        \[\psi_g^{-1}([l])=-\sum_{h \leq g} d_h \sigma_h \otimes 1-\sum_{h \nleq g} d_h \otimes \sigma_{g \vee h}.\]
    It may occur that there are two different $h, h' \in \mG$ such that $g \vee h=g \vee h'$ but, also in this case, the coefficient of $1 \otimes \sigma_{g \vee h}$ is still negative. It follows that $\psi_g^{-1}([l]) \in \Sigma_{P^g,\mG^g} \times \Sigma_{P_g,\mG_g}$.
    \item Let $l=- \sum_{h \in \mG} d_h \sigma_h$ be an element of the $\sigma$-cone, we have 
        \[\zeta_g^{-1}([l])=-\sum_{h \in \mG} d_h \sigma_{\Bar{h}}.\]
    It may occur that there are two different $h, h' \in \mG$ such that $\Bar{h}=\Bar{h'}$ but, also in this case, the coefficient of $\sigma_{\Bar{h}}$ is still negative. Thus, $\zeta_g^{-1}([l]) \in \Sigma_{\tr_{g}P,\tr_{g}\mG}$.
    \item Let $l=- \sum_{h \in \mG} d_h \sigma_h$ be an element of the $\sigma$-cone, we have
        \[p_a^{-1}([l])=-\sum_{h \in \mG} d_h \sigma_{a \vee h}.\]
    It may occur that there are two different $h, h' \in \mG$ such that $a \vee h=a \vee h'$ but, also in this case, the coefficient of $\sigma_{a \vee h}$ is still negative. It follows that $p_a^{-1}([l]) \in \Sigma_{P_a,\mG_a}$.
\end{enumerate}
\end{proof}

Now we are ready to prove the main theorem.

\begin{theorem} \label{thm:HL_HR}
For every element $\ell$ in the $\sigma$-cone $\Sigma_{P,\mG}$ the conditions $\HL_{\DP(P,\mG)}(\ell)$ and $\HR_{\DP(P,\mG)}(\ell)$ hold.
\end{theorem}
\begin{proof}
We prove the statement by induction on $\lvert \mG \rvert$ and $\cd(\hat{1})$.
The base case is $\lvert \mG \rvert =1$, so $\DP(P,\mG)=\Q[x_{\hat{1}}]/(x_{\hat{1}}^{\cd(\hat{1})})$.
In this case, it is known that $-\lambda x_{\hat{1}}$ satisfies Hard Lefschetz and Hodge-Riemann for all positive $\lambda$.

For the inductive step consider a polymatroid $P$, a building set $\mG$, and an element $\ell \in \Sigma_{P,\mG}$.
Under the morphisms of Lemma \ref{lemma:cone} Item \ref{lemma:zeta_g_cone} $\ell$ is mapped in $\Sigma_{\tr_g P, \tr_g \mG}$ for all $g \in \mG$.
Therefore by the inductive hypothesis the image of $\ell$ in $\DP(P,\mG)/\Ann(\sigma_g)$ satisfies Hodge-Riemann relations for all $g\in \mG$.
Notice also that $\ell$ is a sum of $- \sigma_g$ with positive coefficients.
By \cite[Proposition 7.15]{AHK}, $\HL_{\DP(P,\mG)}(\ell)$ holds.

We want to prove that the Hodge-Riemann relations hold for all $\ell \in \Sigma_{P,\mG}$. By \cite[Proposition 7.16]{AHK} it is enough to prove $\HR_{\DP(P,\mG)}(\ell)$ for some $\ell \in \Sigma_{P,\mG}$.
We apply Theorem \ref{thm:HR_inductive_step}: consider any atom $a\in E$, since $\lvert \mG \rvert >1$ then $a \neq \hat{1}$.
Set $C=\DP(P(a),\mG(a))$ and $p(x)=\sum_{i=0}^{\cd(a)} \binom{\cd(a)}{i} x^i (x_a-\sigma_a)^{\cd(a)-i}$; Lemma \ref{lemma:iso_p_a} ensures that $B=\DP(P_a,\mG_a)$ and Lemma \ref{lemma:A_is_correct} that $A=\DP(P,\mG)$.
Let $\ell \in \Sigma_{P(a),\mG(a)}$, then under the morphism $C\to B$ (Lemma \ref{lemma:cone} Item \ref{lemma:p_a_cone}) the class $\ell$ is mapped in $\Sigma_{P_a,\mG_a}$.
By the inductive hypothesis we have $\HR_{\DP(P(a),\mG(a))}(\ell)$ and $\HR_{\DP(P_a,\mG_a)}(\ell)$, hence by Theorem \ref{thm:HR_inductive_step} $\HR_{\DP(P,\mG)}(\ell-\epsilon x_a)$ holds for sufficiently small $\epsilon>0$.

Moreover if $\epsilon$ is small enough then $\ell-\epsilon x_a$ belongs to $\Sigma_{P,\mG}$.
Indeed using the M\"obius inversion formula we have 
\[x_a = \sum_{g \geq a} \mu_{\mG}(a,g) \sigma_g \]
(where we consider $\mG$ as a sub-poset of $L$).
Let $\ell= -\sum_{g \in \mG} d_g \sigma_g$, taking $\epsilon$ smaller than \[\min_{g\geq a} \left\{ \Big\lvert \frac{d_g}{\mu_{\mG}(a,g)} \Big\rvert \right\},\]
then $\ell-\epsilon x_a \in \Sigma_{P,\mG}$.
This concludes the proof.
\end{proof}

\begin{remark}
The ample cone depends on the geometric realization, however our $\sigma$-cone is contained in the ample cone of every realization. Indeed, consider $3$ distinct lines in $\C^3$ and let $P$ be the polymatroid realized by this subspace arrangement.
The projective wonderful model is the blowup of $\mathbb{P}^2$ in $3$ distinct points; there are two cases.
If the three points are collinear the ample cone coincides with the $\sigma$-cone.
Otherwise the three points are in general position and the ample cone is
\[\{-d_{\hat{1}}x_{\hat{1}} -d_ax_a -d_bx_b-d_cx_c \mid d_{\hat{1}}>d_a+d_b, d_{\hat{1}}>d_a+d_c, d_{\hat{1}}>d_b+d_c  \} \]
which strictly contains the $\sigma$-cone.
\end{remark}

\begin{remark}
If we restrict to the case of matroids with arbitrary building sets, the generator $x_{\hat{1}}$ can be eliminated using the relation $x_{\hat{1}}=- \sum_{g\geq e, \, g \neq \hat{1}} x_g$ for any $e \in E$.
Thus the Hard Lefschetz theorem (and so the Hodge-Riemann relations) can be proven for the entire ample cone using as generators $\{ x_g \}_{g \neq \hat{1}}$ instead of $\{\sigma_g\}_{g \in \mG}$ and Lemma \ref{lemma:iso_psi_g} instead of Lemma \ref{lemma:zeta_g}.
\end{remark}

\section{The relative Lefschetz decomposition} \label{sect:rel_Lef}
In this section we provide a decomposition of $\DP(P,\mG)$ as $\DP(P\setminus a, \mG \setminus a)$-module.
This is analogous to the semi-small decomposition of \cite{semismall}, but in this more general setting the corresponding map is not always semi-small.

Indeed, consider an arrangement of hyperplanes $\mathcal{A}$ and the deleted arrangement $\mathcal{A}'=\mathcal{A} \setminus \{H\}$ for some hyperplane $H \in \mathcal{A}$.
There is a projection map between the wonderful models $Y_{\mathcal{A}} \to Y_{\mathcal{A}'}$ (constructed using the maximal building sets).
This map is semi-small and induces the semi-small decomposition of the Chow ring.

In the case of subspace arrangements, the projection between the wonderful models exists but is not semi-small, because the dimension of the fiber of the blow up is too big.
Therefore, the proof of the K\"ahler package done in \cite{semismall} for matroids cannot be adapted to polymatroids.

Recall that for a polymatroid $P=(E,\cd)$ an \textit{atom} $a\in E$ is an element such that the interval $(\hat{0}, \overline{a}) \subset L$ is empty (where $\overline{a}$ is the closure of $a$).
\begin{definition}

For an atom $a$ define the polymatroid $P\setminus a$ on the ground set $E\setminus \{a\}$ with the restricted codimension function $\cd$.
The building set $\mG \setminus a$ is the intersection of $\mG$ with the poset of flats of $P \setminus a$.
\end{definition}

Define a map 
\[\theta_a \colon \DP(P\setminus a, \mG \setminus a) \to \DP(P,\mG)\] 
by $\theta_a(\sigma_h)=\sigma_{\overline{h}}$ where $\overline{h}$ is the closure of $h$ in $P$.
Define the subalgebra $\DPi= \im (\theta_a)$.
\begin{lemma}
The map $\theta_a$ is injective.
\end{lemma}
\begin{proof}
Consider a standard monomial $\sigma_S^b \in \DP(P\setminus a, \mG \setminus a)$ and let $\overline{S}=\{ \overline{h} \mid h \in S\}$.
We have $\theta_a(\sigma_S^b)= \sigma_{\overline{S}}^b$ and it is enough to prove that $\sigma_{\overline{S}}^b$ is a standard monomial.
Notice that $\overline{h}\vee \overline{g}=\overline{h \vee g}$ and the map between the two poset of flats is an inclusion.
Therefore $\overline{S}$ is $\mG$-nested. 
Since $\cd(h)=\cd(\overline{h})$, then $\sigma_{\overline{S}}^b$ is a standard monomial.
\end{proof}

Let $S_a = \{g \in \mG \mid a \in g \textnormal{ and } g \setminus \{a\} \in L \}$ be the set of all flats such that $a$ is a coloop for that flat.

\begin{remark}
Notice that $\theta_a(x_g)=x_g+x_{g\cup \{a\}}$, where we use the convention that $x_h=0$ if $h$ is not a flat of $P$.
Moreover $\DPi$ is generated as an algebra by $\sigma_g$ with $g \notin S_a$ and as vector space by the monomials $\sigma_S^b$ with $S \cap S_a= \emptyset$.
\end{remark}

For $f \in S_a$ define $\DP_{f}$ as the $\DPi$-submodule of $\DP(P,\mG)$ generated by $x_f, x_f^2, \dots x_f^{n_f}$, where 
\[ n_f=  \cd(f)- \cd(f \setminus \{a\})-1 + \lvert F(P, \mG,f\setminus \{a\})\rvert .\]

For a graded module $M= \oplus_i M^i$ we define $M[k]$ to be the graded module such that $(M[k])^i= M^{i+k}$.

\begin{theorem}\label{thm:Lef_dec}
Let $a$ be an atom, then:
\begin{align}
    & x_f^k \DPi [-k] \cong \DP((P\setminus a)^{f\setminus a}, (\mG \setminus a)^{f\setminus a}) \otimes \DP(P_f, \mG_f), \label{eq:struttura_x_f}\\
    & \DP_f= \bigoplus_{k=1}^{n_f} x_f^k \DPi, \label{eq:struttura_DP_f} \\
    &\DP(P,\mG) = \DPi \oplus \bigoplus_{f \in S_a} \DP_f \label{eq:orthogonal decomposition} .
\end{align}
as $\DPi$-modules.
Moreover, the last decomposition is orthogonal with respect to the Poincaré pairing, with the exception of the summand $\DPi$ and $\DP_{\hat{1}}$ (if $a$ is a coloop).
\end{theorem}
For an example of the application of Theorem \ref{thm:Lef_dec} see Section \ref{sect:example}.
Before the proof of the above theorem we need some lemmas.

\begin{lemma} \label{lemma:x_f_sigma_f}
For all $f\in S_a$ and $k\leq n_f$ we have
\begin{equation}
\label{eq:ann_x_sigma}
    x_f \sigma_f^{k-1}\DPi [-k] \cong \DP((P\setminus a)^{f\setminus a}, (\mG \setminus a)^{f\setminus a}) \otimes \DP(P_f, \mG_f),
\end{equation}
and these modules are in direct sum in $\DP(P,\mG)$.
\end{lemma}
\begin{proof}
Notice that for $k\leq n_f$
\[ \DP((\tr_f^{k-1}(P^f))\setminus a, (\tr_f^{k-1}(\mG^f))\setminus a) = \DP((P\setminus a)^{f\setminus a}, (\mG \setminus a)^{f\setminus a}).\]
Using Lemmas \ref{lemma:iso_psi_g} and \ref{lemma:zeta_g} we obtain the isomorphism
\[ x_f \sigma_f^{k-1} \DP(P,\mG) [-k] \simeq \DP(\tr_f^{k-1}(P^f), \tr_f^{k-1}(\mG^f)) \otimes \DP(P_f, \mG_f).\]
It is easy to check that the above isomorphism restricts to the one in eq.\ \eqref{eq:ann_x_sigma}.
For the second claim suppose that there exists a linear combination
\[ \sum_{k=l}^{n_f} \sigma_f^{k-1}p_k =0, \]
for some $p_k \in \DP((P\setminus a)^{f\setminus a}, (\mG \setminus a)^{f\setminus a})$ with $p_l \neq 0$.
The above equality implies 
\[ \sum_{k=l}^{n_f} \sigma_f^{k-l}p_k =0\]
in $\DP(\tr_f^{l-1} P^f, \tr_f^{l-1} \mG^f)$. 
Therefore $p_l$ belongs to the ideal generated by $\sigma_f$ (where $f$ is the top element in the poset of flats of $\tr_f^{l-1} P^f$).
The ideal $(\sigma_f)$ is linearly generated by all monomials $\sigma_T^b$ with $f \in T$.
This yields a contradiction since $p_l$ lies in $\DP((P\setminus a)^{f\setminus a}, (\mG \setminus a)^{f\setminus a})$, which does not contain the generator $\sigma_f$.
\end{proof}

\begin{lemma}\label{lemma:DP_f_modulo_libero}
For all elements $f,g \in S_a$ such that $f \not \geq g$ we have $x_f\sigma_g=x_f \sigma_{g\setminus\{a\}}$.
Moreover, we have 
\[ \DP_f = \bigoplus_{k=1}^{n_f} x_f \sigma_f^{k-1} \DPi. \let\qedsymbol\openbox\qedhere \]
\end{lemma}
\begin{proof}
Consider $h \in \mG$ such that $h\geq g \setminus \{a\}$ and $h\not \geq g$, we need to prove that $x_fx_h=0$.
Notice that $\{ f,h\}$ is an antichain, $f\vee g \in \mG$ and so $(f \vee g) \vee h \in \mG$ because $g\setminus \{a\} \neq \hat{0}$.
Therefore 
\[f\vee h= (f \vee a) \vee ((g\setminus \{a\}) \vee h)=f \vee g  \vee h \in \mG,\]
$\{f,h \}$ is not $\mG$-nested, and $x_fx_h=0$.

For the second statement it is sufficient to prove that $x_f\sigma_f^{k-1} = x_f^k + z$ with some $z \in \sum_{j=1}^{k-1} x_f^j \DPi$.
Write $\sigma_f=x_f + \sum_{g>f} b_g \sigma_g$ for some coefficients $b_g \in \Z$, then
\[ x_f\sigma_f = x_f^2 + x_f\sum_{\substack{g>f \\ g \notin S_a}} b_g \sigma_g + x_f\sum_{\substack{g>f \\ g \in S_a}} b_g \sigma_{g\setminus \{a\}},\]
and all the summands (except $x_f^2$) belong to $x_f\DPi$.
An inductive argument on the exponent $k$ concludes the proof.
\end{proof}

\begin{lemma} \label{lemma:generates}
The submodules $\DPi$ and $\DP_f$ for all $f \in S_a$ generate $\DP(P,\mG)$.
\end{lemma}
\begin{proof}
We prove that each monomial $\sigma_S^b$ belongs to the submodule $M := \DPi + \sum_{F \in S_a} \DP_F$ by complete induction on $f=\min (S \cap S_a)$ and on $b(f)$.

The base case is $S \cap S_a= \emptyset$ and so $\sigma_S^b \in \DPi$.
For the inductive step notice that $S \cap S_a$ is $\mG$-nested, so it is a chain.
Call $f=\min (S \cap S_a)$ and suppose that all monomials $\sigma_S^{b'}$ with $b'(f)< b(f)$ and all monomials $\sigma_{S'}^{b'}$ with $\min(S'\cap S_a)>f$ belong to $M$.

Let $\{g_1, \dots, g_l\}=F(P,\mG,f\setminus \{a\})$ be the set of $\mG$-factors of $f \setminus \{a\}$.
The relation
\[ \sigma_f^{\cd(f)-\cd(f\setminus\{a\})} \prod_{i=1}^k( \sigma_{g_i} - \sigma_{f})=0,\]
holds and in the case $b(f)>n_f$ we can rewrite $\sigma_f^{b(f)}\sigma_{T\setminus \{f\}}$ as sum of monomials with $b'(f)< b(f)$ using the above relation and the fact that $g_i \not \in S_a$. 

In the case $b(f)\leq n_f$ we have
\[\sigma_S^b = x_f \sigma_f^{b(f)-1} \sigma_{S\setminus \{f\}}^b + (\sigma_f - x_f) \sigma_f^{b(f)-1} \sigma_{S\setminus \{f\}}^b. \]

Using the first assertion of Lemma \ref{lemma:DP_f_modulo_libero}, it follows that the element $x_f \sigma_f^{b(f)-1} \sigma_{S\setminus \{f\}}^b$ belongs to $x_f \sigma_f^{b(f)-1} \DPi \subset M$.
The second summand $(\sigma_f - x_f) \sigma_f^{b(f)-1} \sigma_{S\setminus \{f\}}^b$ is a linear combination of monomials $\sigma_h\sigma_f^{b(f)-1} \sigma_{S\setminus \{f\}}^b$ with $h>f$ and so belongs to $M$ by the inductive hypothesis.
\end{proof}

\begin{proof}[Proof of Theorem \ref{thm:Lef_dec}]
As in Remark \ref{rem:top_in_G}, we may assume that $\hat{1}\in \mG$.
By Lemmas \ref{lemma:DP_f_modulo_libero} and \ref{lemma:x_f_sigma_f}, $\DP_f$ is a free $\DP((P\setminus a)^{f\setminus a}, (\mG \setminus a)^{f\setminus a}) \otimes \DP(P_f, \mG_f)$-module with basis $x_f\sigma_f^{k-1}$ for $k=1, \dots, n_f$.
The elements $\{x_f^{k}\}_k$ written in the basis  $\{x_f\sigma_f^{k-1}\}_k$ form an upper triangular matrix with ones on the diagonal (the inverse of the one given in Lemma \ref{lemma:DP_f_modulo_libero}). 
Eq.\ \eqref{eq:struttura_x_f} and eq.\ \eqref{eq:struttura_DP_f} follow.

In order to prove eq.\ \eqref{eq:orthogonal decomposition} we first prove the orthogonality.
Let $f\neq \hat{1}$; the elements $\DP_f$ and $\DPi$ are orthogonal because the product is contained in $\DP_f$ which is  zero in degree $\cd(\hat{1})-1$.
Indeed from eq.\ \eqref{eq:struttura_x_f} and eq.\ \eqref{eq:struttura_DP_f}, it follows that the top degree of $\DP_f$ is $\cd(\hat{1})-2$.

Consider generic elements $x_f^by \in \DP_f$ and $x_g^c z \in \DP_g$ in complementary degrees (with $y,z \in \DPi$).
The product is zero if $f$ and $g$ are incomparable.
Otherwise, by symmetry we may assume $g>f$, hence
\[ x_fx_g^c = x_f(x_g+x_{g\setminus \{a\}})^c.\]
Since $x_g+x_{g\setminus \{a\}} \in \DPi$, we obtain that the product lie in $\DP_f$.
Again the top degree is zero since $f \neq \hat{1}$.

We prove that if $a$ is a coloop then $\DPi \cap \DP_{\hat{1}}=0$.
In that case $\DP_{\hat{1}}$ is the ideal generated by $\sigma_{\hat{1}}$. 
This ideal is linearly generated by all standard monomials $\sigma_S^b$ with $\hat{1} \in S$.
Since $\hat{1}\in S_a$ then $\DPi \cap \DP_{\hat{1}}=0$.
The direct sum of eq.\ \eqref{eq:orthogonal decomposition} follows from the orthogonality of all other summands together with Lemma \ref{lemma:generates} and Theorem \ref{thm:Poinc_duality}.
\end{proof}

\section{Characteristic polynomial} \label{sect:char_poly}
In this section we study the coefficients of the (reduced) characteristic polynomial of a polymatroid. 

We consider only maximal building sets, so we omit the building set from the notations.
Moreover we suppose that the polymatroid is without loops, i.e.\ $\cd(\{e\})>0$ for all $e \in E$.

Let $\alpha= \alpha_P = -x_{\hat{1}}$ and $\beta=\beta_P= \sum_{g \in \mG_{\max}} x_g$ be two elements in $\DP^1(P)$.
We denote by $\mu_L(a,b)$ the M\"obius function of $L$.

\begin{lemma} \label{lemma:rec_relation}
For any polymatroid $P$ with $\cd(E)>0$ and $r=\cd(E)-1$ we have
\[ \deg (\beta_P^r)= (-1)^r + \sum_{g \in L \setminus \{\hat{0},\hat{1}\}} (-1)^{\cd(g)-1} \deg(\beta_{P_g}^{r-\cd(g)}). \let\qedsymbol\openbox\qedhere\]
\end{lemma}
\begin{proof}
A flag with repetition is $\mathcal{F}=(F_1^{a_1} \subsetneq F_2^{a_2} \subsetneq \dots \subsetneq F_l^{a_l})$ where $a_i >0$ are the multiplicity of the flats $F_i \in L$.
We also require that $\sum_{i=1}^l a_i=r$.
Define $x_{\mathcal{F}}=\prod_{i=1}^{\lvert \mathcal{F} \rvert} x_{F_i}^{a_i}$, we will prove that $x_{\mathcal{F}}=0$ if $\cd(F_1)>a_1$.
More generally we have $x_{\mathcal{F}}=0$ if $\cd(F_i)> \sum_{j=1}^i a_j$ for some $i$, but we prove and use the implication only for $i=1$.
From the isomorphism $\psi_g$ of Lemma \ref{lemma:iso_psi_g} we obtain
\[ x_{\mathcal{F}}= x_{F_1}\psi_{F_1}((x_{F_1} \otimes 1 - 1 \otimes \beta_{P_g})^{a_1-1}(1 \otimes x_{\mathcal{F}'})), \]
where $\mathcal{F}'=(F_2^{a_2} \subsetneq \dots \subsetneq F_l^{a_l})$.
Notice that the degree of $x_{\mathcal{F}'}$ is $r-a_1$, which is greater than $r-\cd(F_1)$, the top degree of $\DP(P_g)$.  

Let $\binom{r}{a}$ be the multinomial coefficient where $a=(a_1,\dots, a_l)$ and $\sum_{i=1}^l a_i=r$.
Since $x_fx_g=0$ if $f$ and $g$ are incomparable, we have
\begin{align*}
    \beta_P^r &= \sum_{\mathcal{F} \textnormal{ flag of } P} \binom{r}{a} x_{\mathcal{F}} \\
    &= \sum_{F \in L \setminus \{\hat{0}\}}  \sum_{\substack{\mathcal{F} \textnormal{ flag of } P \\ F_1=F}} \binom{r}{a} x_{\mathcal{F}} \\
    &= \sum_{F \in L \setminus \{\hat{0}\}}  \sum_{k =\cd(F)}^{r} \binom{r}{k} x_F^k \sum_{\mathcal{F}' \textnormal{ flag of } P_F} \binom{r-k}{a'} x_{\mathcal{F}'} \\
    &= \sum_{F \in L \setminus \{\hat{0}\}}  \sum_{k =\cd(F)}^{r} \binom{r}{k} x_F^k \beta_{P_F}^{r-k}.
\end{align*}
The summand relative to $F=\hat{1}$ is exactly $x_{\hat{1}}^r$ and contributes $(-1)^r$ to $\deg(\beta_P^{r})$.
It is enough to prove
\[ \deg \Bigl( \sum_{k =\cd(g)}^{r} \binom{r}{k} x_g^k \beta_{P_g}^{r-k} \Bigr) = \deg (\beta_{P_g}^{r-\cd(g)}) \]
for every $g \in L \setminus \{\hat{0},\hat{1}\}$.
We use Lemma \ref{lemma:iso_psi_g} to obtain:
\begin{align*}
    \deg &\Bigl( \sum_{k =\cd(g)}^{r} \binom{r}{k} x_g^k \beta_{P_g}^{r-k} \Bigr)= \\
    &= \sum_{k =\cd(g)}^{r} \binom{r}{k} \deg ((x_g \otimes 1 - 1 \otimes \beta_{P_g})^{k-1}(1 \otimes \beta_{P_g}^{r-k})) \\
    &= \sum_{k =\cd(g)}^{r} (-1)^{k-\cd(g)} \binom{r}{k} \binom{k-1}{\cd(g)-1} \deg (x_g^{\cd(g)-1} \otimes \beta_{P_g}^{r-\cd(g)}) \\
    &= \sum_{k =\cd(g)}^{r} (-1)^{k-1} \binom{r}{k} \binom{k-1}{\cd(g)-1} \deg (\beta_{P_g}^{r-\cd(g)}) \\
    &= (-1)^{\cd(g)-1}\deg (\beta_{P_g}^{r-\cd(g)}),
\end{align*}
where in the last equality we used the identity
\[\sum_{k =\cd(g)}^{r} (-1)^{k} \binom{r}{k} \binom{k-1}{\cd(g)-1} = (-1)^{\cd(g)} \]
which follows from \cite[eq.\ 5.24]{CM} with $l=r$, $m=0$, $n=\cd(g)-1$, and $s=-1$.
\end{proof}

\begin{lemma} \label{lemma:deg_beta}
For every polymatroid $P$ with poset of flats $L$ and $r=\cd(E)-1$ with $\cd(E)>0$ we have 
\[ \deg (\beta_P^{r}) = (-1)^{\cd(E)} \mu_L (\hat{0},\hat{1}). \let\qedsymbol\openbox\qedhere \]
\end{lemma}
\begin{proof}
It is known that $\mu_L (\hat{0},\hat{1}) = \tilde{\chi}(\Delta(\hat{0},\hat{1}))$, i.e.\ the M\"obius function coincides with the reduced Euler characteristic of the order complex of the poset $L \setminus \{\hat{0},\hat{1}\}$ (e.g.\ see \cite{Rota}).
Let $L^{\op}$ be the opposite (dual) lattice of $L$ which is defined on the same set of $L$ but with reversed order, i.e., $x \leq y$ in $L^{\op}$ if and only if $y \leq x$ in $L$.
Since the order complexes of $L$ and $L^{\op}$ are the same simplicial complex, we have
\[ \mu_L (\hat{0},\hat{1})= \mu_{L^{\op}} (\hat{0},\hat{1}).\]
Define $\deg(\beta_P^0)=1$ for rank zero polymatroids $P$.
Therefore, the functions $(-1)^{\cd(E)}\deg (\beta_P^{r})$ and $\mu_{L^{\op}} (\hat{0},\hat{1})$ satisfy the same recurrence relation.
One is given by the definition of $\mu_{L^{\op}} (\hat{0},\hat{1})$ and the other by Lemma \ref{lemma:rec_relation}.
This concludes the proof.
\end{proof}

Notice that if $L$ is a geometric lattice (i.e.\ the poset of flats of a matroid), then the M\"obius function has alternating sign, hence in this case $\deg(\beta^r)\in \N_0$.

\begin{definition}
The \textit{characteristic polynomial} of a polymatroid $P$ is
\[ \chi_P(\lambda)= \sum_{g \in L} \mu_L(\hat{0},g) \lambda^{\dim(g)},\]
where $\dim (g)= \cd(\hat{1})- \cd(g)$.
Since $\chi_P(1)=0$ by the definition of M\"obius function, we define the \textit{reduced characteristic polynomial} as 
\[ \overline{\chi}_P(\lambda)= \frac{\chi_P(\lambda)}{\lambda-1}. \let\qedsymbol\openbox\qedhere\]
\end{definition}

This definition of reduced characteristic polynomial coincides with the one in the introduction eq.\ \eqref{eq:def_chi_intro}, as stated in \cite{Whittle}.

\begin{theorem} \label{lemma:chi_deg}
For every polymatroid $P$, we have
\[ \overline{\chi}_P(\lambda) = \sum_{i=0}^{r} (-1)^{i} \deg_P (\alpha_P^i \beta_P^{r-i}) \lambda^i\]
where $r=\cd(E)-1$.
\end{theorem}
\begin{proof}
We show that $\overline{\chi}_P(\lambda)$ and the right hand side satisfy the same recurrence:
\[ q_P(\lambda) - \lambda q_{\tr_{\hat{1}} P}(\lambda) = - \mu_L(\hat{0}, \hat{1})\]
where $L$ is the poset of flats of $P$.

Let $\tr_{\hat{1}} L$ be the poset of flats of $\tr_{\hat{1}} P$ and notice that $\mu_L(\hat{0},g)= \mu_{\tr_{\hat{1}} L}(\hat{0},g)$ for all $g$ such that $\dim (g)>1$.
Therefore $\chi_P(\lambda)- \lambda \chi_{\tr P}(\lambda)$ is a polynomial of degree one divisible by $\lambda-1$.
Hence $\overline{\chi}_P(\lambda) - \lambda \overline{\chi}_{\tr P}(\lambda)$ is constant and equal to $ \overline{\chi}_P(0)= - \mu_L(\hat{0},\hat{1})$.
This proves that $\overline{\chi}_P(\lambda)$ satisfies the recurrence.

Now observe that for $i>0$ $\deg_P (\alpha_P^i \beta_P^{r-i})= \deg_{\tr P} (\alpha_{\tr P}^{i-1} \beta_{\tr P}^{r-i})$ by Lemma \ref{lemma:zeta_g}.
This proves that 
\[ \begin{split}
    \sum_{i=0}^{r} (-1)^i \deg_P (\alpha_P^i \beta_P^{r-i}) \lambda^i - \lambda \sum_{i=0}^{r-1} (-1)^i \deg_{\tr P} (\alpha_{\tr P}^{i-1} \beta_{\tr P}^{r-i}) \lambda^i= \\= (-1)^r \deg_P(\beta^r),
\end{split}\]
and so Lemma \ref{lemma:deg_beta} proves the recurrence.

The base case $\cd(E)=1$ is trivial, so the proof follows.
\end{proof}
For an explicit example check Section \ref{sect:example}.
\begin{corollary}
The coefficient of $\lambda^i$ of the reduced characteristic polynomial $\overline{\chi}_P(\lambda)$ is (up to the sign) the reduced Euler characteristic of the order complex of the poset $(\tr_{\hat{1}}^{i} L) \setminus \{ \hat{0}, \hat{1}\}$: 
\[ [\lambda^i]\overline{\chi}_P(\lambda)= (-1)^{\cd(E)}
\tilde{\chi}(\Delta((\tr_{\hat{1}}^{i} L) \setminus \{ \hat{0}, \hat{1}\})). \let\qedsymbol\openbox\qedhere \]
\end{corollary}
\begin{proof}
It follows from Theorem \ref{lemma:chi_deg} and Lemma \ref{lemma:deg_beta}.
\end{proof}

\begin{remark} \label{remark:not_log_concave}
The coefficients of the characteristic polynomial $\chi_P$ and of the reduced characteristic polynomial $\overline{\chi}_P$ do not form a log-concave sequence.
Indeed if $P_1$ is the polymatroid associated to $4$ subspaces of codimension $2,3,4,4$ in $\C^5$ in general position, then
\[ \chi_{P_1}(\lambda)= \lambda^5 - \lambda^3 -\lambda^2 -2\lambda +3,\]
which is not log-concave.
Let $P_2$ be the polymatroid on $E=\{ a,b,c,d,e \}$ with rank defined by $\cd(a)=2$, $\cd(b)=3$, $\cd(c)=4$, $\cd(d)=4$, $\cd(e)=1$, by $\cd(A)=6$ if $\lvert A \rvert \geq 3$ and $\cd(\{x,y\})= \min \{5, \cd(x)+\cd(y)\}$.
The reduced characteristic polynomial is not log-concave because
\[ \overline{\chi}_{P_2}(\lambda)= \lambda^5-\lambda^3-\lambda^2 -2\lambda +6. \let\qedsymbol\openbox\qedhere\]
\end{remark}

\section{An example} \label{sect:example}
Let $E=\{a,b,c\}$ and $\cd \colon 2^E \to \mathbb{N}$ the function defined by 
\begin{align*}
     &\cd(a)=\cd(b)=2, \quad \quad \cd(ab)=\cd(c)=4, \\ 
    &\cd(ac)=\cd(bc)=\cd(abc)=5.
\end{align*}
This function defines a polymatroid $P$ with poset of flats $L$ shown in Figure \ref{fig:Hasse}.
Near every cover relation, the relative codimension of the two flats is shown.
\begin{figure}
    \centering
    \begin{tikzpicture}
         \node (top) at (0,3) {$\hat{1}$};
         \node (bottom) at (0,0) {$\hat{0}$};
         \node (a) at (-1,1) {$a$};
         \node (b) at (0,1) {$b$};
         \node (c) at (1,1.5) {$c$};
         \node (ab) at (-0.5,2) {$ab$};
         \draw (ab) -- (a) node [midway, left] {$\scriptstyle 2$};
         \draw (a) -- (bottom) node [midway, left] {$\scriptstyle 2$};
         \draw (bottom) -- (b) node [midway, left] {$\scriptstyle 2$};
         \draw (b) -- (ab) node [midway, right] {$\scriptstyle 2$};
         \draw (ab) -- (top) node [midway, left] {$\scriptstyle 1$};
         \draw (c) -- (bottom) node [midway, right] {$\scriptstyle 4$};
         \draw (c) -- (top) node [midway, right] {$\scriptstyle 1$};
    \end{tikzpicture}
    \caption{The Hasse diagram of the poset of flats $L$ of \Cref{sect:example}}
    \label{fig:Hasse}
\end{figure}
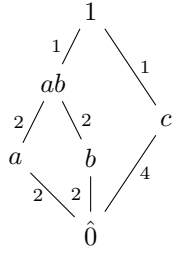
This polymatroid is realizable: a realization is the collection in $\C^5$ of two subspace of dimension $3$ and a line in general position.

Consider the (minimal) building set $\mG=\{a,b,c, \hat{1}\}$; the nested set complex $n(P,\mG)$ is shown in \Cref{fig:nested}.
\begin{figure}
    \centering
     \begin{tikzpicture}
     \tikzstyle{point}=[circle,thick,draw=black,fill=black,inner sep=0pt,minimum width=4pt,minimum height=4pt]
         \node[label=above:{$\hat{1}$}] (uno)[point] at (0,1) {};
         \node[label=left:{$a$}] (a)[point] at (-2,2) {};
         \node[label=left:{$b$}] (b)[point] at (-2,0) {};
         \node[label=right:{$c$}] (c)[point] at (2,1) {};
         \fill[color=gray!60] (a.center) -- (b.center) -- (uno.center);
         \draw (uno) -- (a) -- (b) -- (uno) -- (c);
    \end{tikzpicture}
    \caption{The nested set complex $n(P,\mG)$.}
    \label{fig:nested}
\end{figure}
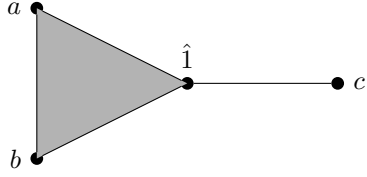

The algebra $B(P,\mG)$ is generated by $x_a,x_b,x_c,x_{\hat{1}}, e_a,e_b,e_c,e_{\hat{1}}$ with relations:
\begin{align*}
    & e_a e_c=e_b e_c=0 & & x_a x_c=x_b x_c=0 \\
    & x_a e_c=x_b e_c=0 & & e_a x_c=e_b x_c=0 \\
    & (x_a+x_{\hat{1}})^2= (x_b+x_{\hat{1}})^2=0 & & (x_c+x_{\hat{1}})^4=0 \\
    & x_{\hat{1}}^5=0 & & x_c x_{\hat{1}}= e_c x_{\hat{1}}= 0 \\
    & x_a x_{\hat{1}}^3= e_a x_{\hat{1}}^3= 0 & & x_b x_{\hat{1}}^3= e_b x_{\hat{1}}^3= 0 \\
    & x_a x_b x_{\hat{1}}= e_a x_b x_{\hat{1}}= 0 & & e_a e_b x_{\hat{1}}= x_a e_b x_{\hat{1}}= 0
\end{align*}
The homogeneous component $B^{4,1}(P,\mG)$ has dimension $12$ and the additive basis provided by Corollary \ref{cor:base_additiva_e_x} is:
\begin{align*}
    &e_{\hat{1}}x_ax_b, \, e_{\hat{1}}x_ax_{\hat{1}}, \, e_{\hat{1}}x_b x_{\hat{1}}, \, e_{\hat{1}} x_c^2, \, e_{\hat{1}}x_{\hat{1}}^2, \\
    & e_ax_ax_b, \, e_ax_ax_{\hat{1}}, \, e_a x_{\hat{1}}^2, \, e_bx_ax_b, \, e_bx_ax_{\hat{1}}, \, e_b x_{\hat{1}}^2, \, e_c x_c^2.
\end{align*}

\begin{table}
    \centering
    \begin{tabular}{c|ccccc}
         \tiny{\textbf{2}} & 1 & 2 & 1  \\
         \tiny{\textbf{1}} & 3 & 7 & 7 & 3 \\
         \tiny{\textbf{0}} & 1 & 4 & 5 & 4 & 1 \\
         \hline
         & \tiny{\textbf{0}} & \tiny{\textbf{1}} & \tiny{\textbf{2}} & \tiny{\textbf{3}} & \tiny{\textbf{4}} 
    \end{tabular}
    \caption{The dimensions of $B^{2p,q}(P,\mG)/(e_1)$ in position $(p,q).$}
    \label{tab:dim_B_mod_e}
\end{table}

\begin{table}
    \centering
    \begin{tabular}{c|ccccc}
         \tiny{\textbf{3}} & 1 & 2 & 1  \\
         \tiny{\textbf{2}} & 4 & 9 & 8 & 3 \\
         \tiny{\textbf{1}} & 4 & 11 & 12 & 7 & 1 \\
         \tiny{\textbf{0}} & 1 & 4 & 5 & 4 & 1 \\
         \hline
         & \tiny{\textbf{0}} & \tiny{\textbf{1}} & \tiny{\textbf{2}} & \tiny{\textbf{3}} & \tiny{\textbf{4}}
    \end{tabular}
    \caption{The dimensions of $B^{2p,q}(P,\mG)$ in position $(p,q).$}
    \label{tab:dim_B}
\end{table}

Notice that $B(P,\mG)= B(P,\mG)/(e_{\hat{1}}) \otimes \langle 1, e_{\hat{1}} \rangle$ and their dimensions are reported in \Cref{tab:dim_B_mod_e,tab:dim_B}.

The other presentation of $B(P,\mG)$ is given by generators
$\sigma_a, \sigma_b, \sigma_c, \sigma_{\hat{1}}$, $\tau_a, \tau_b, \tau_c, \tau_{\hat{1}}$ and relations:
\begin{align*}
    & (\tau_a-\tau_{\hat{1}}) (\tau_c-\tau_{\hat{1}})= 0 && (\tau_b-\tau_{\hat{1}}) (\tau_c-\tau_{\hat{1}})=0  & &  \sigma_c^4=0 \\
    & (\sigma_a-\sigma_{\hat{1}}) (\sigma_c-\sigma_{\hat{1}})=0 & & (\sigma_b-\sigma_{\hat{1}}) (\sigma_c-\sigma_{\hat{1}})=0 & & \sigma_a^2=0 \\
    & (\sigma_a-\sigma_{\hat{1}}) (\tau_c-\tau_{\hat{1}})=0 && (\sigma_b-\sigma_{\hat{1}}) (\tau_c-\tau_{\hat{1}})=0 &  &  \sigma_b^2=0\\
    &(\tau_a-\tau_{\hat{1}}) (\sigma_c-\sigma_{\hat{1}})=0 && (\tau_b-\tau_{\hat{1}}) (\sigma_c-\sigma_{\hat{1}})=0 && \sigma_{\hat{1}}^5=0 \\
    & (\sigma_c - \sigma_{\hat{1}}) \sigma_{\hat{1}}=0 && (\tau_c - \tau_{\hat{1}}) \sigma_{\hat{1}}=0 & & (\sigma_a - \sigma_{\hat{1}}) \sigma_{\hat{1}}^3=0 \\
    & (\tau_a - \tau_{\hat{1}}) \sigma_{\hat{1}}^3= 0 & & (\sigma_b - \sigma_{\hat{1}}) \sigma_{\hat{1}}^3=0 && (\tau_b - \tau_{\hat{1}}) \sigma_{\hat{1}}^3= 0 \\
    & (\sigma_a - \sigma_{\hat{1}})(\sigma_b - \sigma_{\hat{1}}) \sigma_{\hat{1}}=0 && (\tau_a - \tau_{\hat{1}})(\sigma_b - \sigma_{\hat{1}}) \sigma_{\hat{1}}=0 \\
    &  (\tau_a - \tau_{\hat{1}})(\tau_b - \tau_{\hat{1}}) \sigma_{\hat{1}}=0 & & (\sigma_a - \sigma_{\hat{1}})(\tau_b - \tau_{\hat{1}}) \sigma_{\hat{1}}=0
\end{align*}

The homogeneous component $B^{4,1}(P, \mG)$ has dimension $12$ and the additive basis provided by Corollary \ref{cor:base_additiva_sigma_tau} is:
\begin{align*}
    &\tau_{\hat{1}} \sigma_a \sigma_b, \, \tau_{\hat{1}}\sigma_a\sigma_{\hat{1}}, \, \tau_{\hat{1}}\sigma_b \sigma_{\hat{1}}, \, \tau_{\hat{1}} \sigma_c^2, \, \tau_{\hat{1}}\sigma_{\hat{1}}^2, \\
    & \tau_a\sigma_a\sigma_b, \, \tau_a\sigma_a\sigma_{\hat{1}}, \, \tau_a \sigma_{\hat{1}}^2, \, \tau_b\sigma_a\sigma_b, \, \tau_b\sigma_a\sigma_{\hat{1}}, \, \tau_b \sigma_{\hat{1}}^2, \, \tau_c \sigma_c^2.
\end{align*}
The set of critical monomials is:
\begin{align*}
    & 1, \tau_a \sigma_a, \, \tau_b \sigma_b, \, \tau_c \sigma_c^3, \, \tau_{\hat{1}} \sigma_{\hat{1}}^4, \, \tau_a \tau_b \sigma_a \sigma_b,\, \tau_a \tau_{\hat{1}} \sigma_a \sigma_{\hat{1}}^2,\, \tau_b \tau_{\hat{1}} \sigma_b \sigma_{\hat{1}}^2,
    \tau_c \tau_{\hat{1}} \sigma_c^3, \, \tau_a \tau_b \tau_{\hat{1}} \sigma_a \sigma_b,
\end{align*}
and the dimensions of $\CM^{2p,q}(P,\mG)$ are given in \Cref{tab:es_cmu}.
The rank of the cohomology group of $(B(P,\mG),\dd)$ are given in \Cref{tab:cohom_B}

\begin{table}
    \centering
    \begin{tabular}{c|ccccc}
         \tiny{\textbf{3}}  & 0 & 0 & 1  \\
         \tiny{\textbf{2}}  & 0 & 0 & 1 & 3 \\
         \tiny{\textbf{1}}  & 0 & 2 & 0 & 1 & 1 \\
         \tiny{\textbf{0}}  & 1 & 0 & 0 & 0 & 0 \\
         \hline
         & \tiny{\textbf{0}} & \tiny{\textbf{1}} & \tiny{\textbf{2}} & \tiny{\textbf{3}} & \tiny{\textbf{4}}
    \end{tabular}
    \caption{The dimensions of $\CM(P,\mG)$ in position $(p,q).$}
    \label{tab:es_cmu}
\end{table}

\begin{table}
    \centering
    \begin{tabular}{c|ccccc}
         \tiny{\textbf{3}} & 0 & 0 & 0  \\
         \tiny{\textbf{2}} & 0 & 0 & 1 & 1 \\
         \tiny{\textbf{1}} & 0 & 2 & 0 & 1 & 0 \\
         \tiny{\textbf{0}} & 1 & 0 & 0 & 0 & 0 \\
         \hline
         & \tiny{\textbf{0}} & \tiny{\textbf{1}} & \tiny{\textbf{2}} & \tiny{\textbf{3}} & \tiny{\textbf{4}}
    \end{tabular}
    \caption{The dimensions of $H^{2p,q}(B(P,\mG), \dd)$ in position $(p,q).$}
    \label{tab:cohom_B}
\end{table}
As an example we have
\begin{align*}
  \dd (c\mu(ab\hat{1})) &= \dd(\tau_a \tau_b \tau_{\hat{1}} \sigma_a \sigma_b) = \tau_a \tau_b \sigma_{\hat{1}} \sigma_a \sigma_b \\
  &= \tau_b \tau_{\hat{1}} \sigma_b \sigma_{\hat{1}}^2 - \tau_a \tau_{\hat{1}} \sigma_a \sigma_{\hat{1}}^2 = c\mu(b\hat{1})-c\mu(a\hat{1}),
\end{align*}
that coincides with $\dd((a,b,\hat{1}))=(b,\hat{1})- (a,\hat{1})$ in the differential algebra $\CM(P,\mG)$.
Moreover, in $\CM(P,\mG)$ we have
\[(a)\cdot (b)= \lambda(a,b) - \lambda(b,a) = (a,b),\]
because $a \prec b$ and it corresponds to the equality 
\[ c\mu(a)c\mu(b)= \tau_a \sigma_a \tau_b \sigma_b=\tau_a \tau_b \sigma_a  \sigma_b = c\mu(ab).\]

The posets related to $P$ and $a$ are shown in \Cref{fig:four_Hasse}.
The polymatroids $P(a)$ and $P\setminus a$ are equal by coincidence;
see below for the poset $P(a)$ relative to the maximal building set.

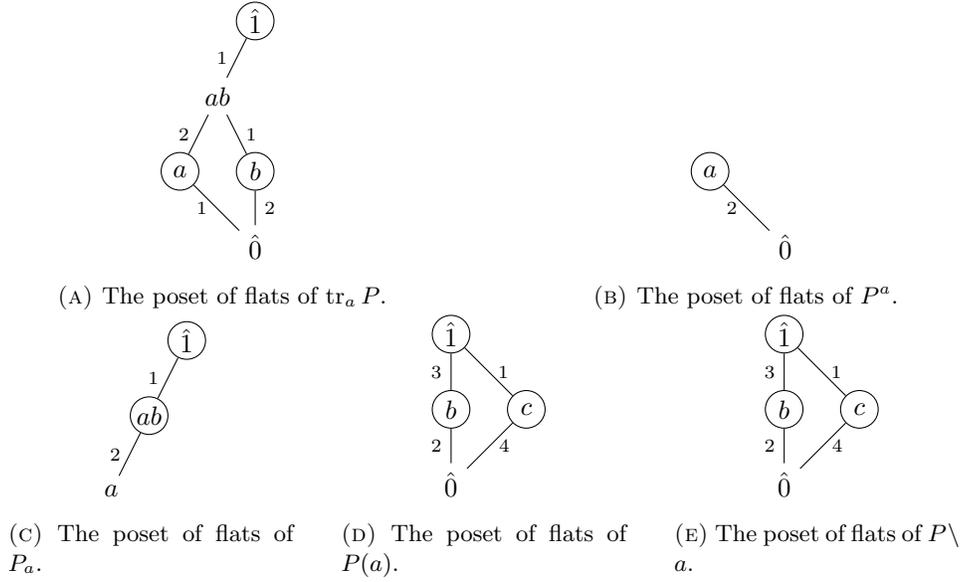
\begin{figure}
    \centering
    \begin{subfigure}[b]{0.45\textwidth}
    \centering
    \begin{tikzpicture}
         \node[circle,draw,minimum size=0.5cm,inner sep=0pt] (top) at (0,3) {$\hat{1}$};
         \node (ab) at (-0.5,2) {$ab$};
         \node (bottom) at (0,0) {$\hat{0}$};
         \node[circle,draw,minimum size=0.5cm,inner sep=0pt] (a) at (-1,1) {$a$};
         \node[circle,draw,minimum size=0.5cm,inner sep=0pt] (b) at (0,1) {$b$};
         \draw (ab) -- (a) node [midway, left] {$\scriptstyle 2$};
         \draw (a) -- (bottom) node [midway, left] {$\scriptstyle 1$};
         \draw (bottom) -- (b) node [midway, right] {$\scriptstyle 2$};
         \draw (b) -- (ab) node [midway, right] {$\scriptstyle 1$};
         \draw (ab) -- (top) node [midway, left] {$\scriptstyle 1$};
    \end{tikzpicture}
    \caption{The poset of flats of $\tr_a P$.}
    \label{fig:tr_a_P}
    \end{subfigure}
    \hfill
    \begin{subfigure}[b]{0.45\textwidth}
    \centering
    \begin{tikzpicture}
         \node[circle,draw,minimum size=0.5cm,inner sep=0pt] (a) at (-1,1) {$a$};
         \node (bottom) at (0,0) {$\hat{0}$};
         \draw (a)--(bottom) node [midway, left] {$\scriptstyle 2$};
    \end{tikzpicture}
    \caption{The poset of flats of $P^a$.}
    \label{fig:P_up_a}
    \end{subfigure}
    
    \begin{subfigure}[b]{0.3\textwidth}
    \centering
    \begin{tikzpicture}
         \node[circle,draw,minimum size=0.5cm,inner sep=0pt] (top) at (0,3) {$\hat{1}$};
         \node (a) at (-1,1) {$a$};
         \node[circle,draw,minimum size=0.5cm,inner sep=0pt] (ab) at (-0.5,2) {$ab$};
         \draw (top) -- (ab) node [midway, left] {$\scriptstyle 1$};
         \draw (ab)--(a) node [midway, left] {$\scriptstyle 2$};
    \end{tikzpicture}
    \caption{The poset of flats of $P_a$.}
    \label{fig:P_down_a}
    \end{subfigure}
    \hfill
    \begin{subfigure}[b]{0.3\textwidth}
    \centering
    \begin{tikzpicture}
         \node[circle,draw,minimum size=0.5cm,inner sep=0pt] (top) at (0,2) {$\hat{1}$};
         \node (bottom) at (0,0) {$\hat{0}$};
         \node[circle,draw,minimum size=0.5cm,inner sep=0pt] (b) at (0,1) {$b$};
         \node[circle,draw,minimum size=0.5cm,inner sep=0pt] (c) at (1,1) {$c$};
         \draw (bottom)--(c) node [midway, right] {$\scriptstyle 4$};
         \draw (c) -- (top) node [midway, right] {$\scriptstyle 1$};
         \draw (top) -- (b) node [midway, left] {$\scriptstyle 3$};
         \draw (b) -- (bottom) node [midway, left] {$\scriptstyle 2$};
    \end{tikzpicture}
    \caption{The poset of flats of $P(a)$.}
    \label{fig:P(a)}
    \end{subfigure}
        \hfill
    \begin{subfigure}[b]{0.3\textwidth}
    \centering
    \begin{tikzpicture}
         \node[circle,draw,minimum size=0.5cm,inner sep=0pt] (top) at (0,2) {$\hat{1}$};
         \node (bottom) at (0,0) {$\hat{0}$};
         \node[circle,draw,minimum size=0.5cm,inner sep=0pt] (b) at (0,1) {$b$};
         \node[circle,draw,minimum size=0.5cm,inner sep=0pt] (c) at (1,1) {$c$};
         \draw (bottom)--(c) node [midway, right] {$\scriptstyle 4$};
         \draw (c) -- (top) node [midway, right] {$\scriptstyle 1$};
         \draw (top) -- (b) node [midway, left] {$\scriptstyle 3$};
         \draw (b) -- (bottom) node [midway, left] {$\scriptstyle 2$};
    \end{tikzpicture}
    \caption{The poset of flats of $P \setminus a$.}
    \label{fig:P_setminus_a}
    \end{subfigure}
    
    \caption{The Hasse diagram of some posets related to $a$. The circled nodes are in the corresponding building sets.}
    \label{fig:four_Hasse}
\end{figure}

The $\sigma$-cone $\Sigma_{P,\mG}$ is given by the linear combinations $-d_a\sigma_a -d_b \sigma_b -d_c\sigma_c- d_{\hat{1}}\sigma_{\hat{1}}$ with positive coefficients $d_g>0$.

We have $\Ann(x_a)=(x_c, x_b \sigma_{\hat{1}}, \sigma_{\hat{1}}^3)$ and so in $\DP(P,\mG)/\Ann(x_a)$ we have $\sigma_c=\sigma_{\hat{1}}$, $(\sigma_b-\sigma_{\hat{1}})\sigma_{\hat{1}}=0$, and $\sigma_{\hat{1}}^3=0$.
The last two equations correspond to the defining relation for $\DP(P_a,\mG_a)$.
Similarly, $\Ann (\sigma_a)=(\sigma_c-\sigma_{\hat{1}}, \sigma_a, \sigma_{\hat{1}}^4, (\sigma_b-\sigma_{\hat{1}}) \sigma_{\hat{1}}^2)$ and these are exactly the equations defining $\DP(\tr_a P, \tr_a \mG)$ that do not appear in $\DP(P,\mG)$.

The relative Lefschetz decomposition with respect to the atom $a$ is 
\[ \DP(P,\mG)= \DPi \oplus x_a \DPi,\]
where 
\[ \DPi = \langle 1, \sigma_b, \sigma_c, \sigma_{\hat{1}}, \sigma_b\sigma_{\hat{1}}, \sigma_c^2, \sigma_{\hat{1}}^2, \sigma_b\sigma_{\hat{1}}^2, \sigma_c^3, \sigma_{\hat{1}}^3, \sigma_{\hat{1}}^4, \rangle\]
and 
\[\DP_a = x_a \DPi  = \langle x_a, x_a \sigma_b, x_a \sigma_{\hat{1}}, x_a \sigma_{\hat{1}}^2 \rangle \simeq \DP(P_a, \mG_a)[1].\]
The relative Lefschetz decomposition with respect to the atom $c$ is 
\[ \DP(P,\mG)= \DP_{(c)} \oplus \DP_{\hat{1}} \oplus \DP_c,\]
where $\DP_{(c)}= \langle 1, \sigma_a, \sigma_b, \sigma_a \sigma_b \rangle$ and the other $\DP_{(c)}$-modules are $\DP_c=\langle x_c, x_c^2, x_c^3 \rangle$ and $\DP_{\hat{1}}=x_{\hat{1}} \DP_{(c)} \oplus x_{\hat{1}}^2 \DP_{(c)}$.
Moreover we have $x_{\hat{1}} \DP_{(c)}[-1] \simeq x_{\hat{1}}^2 \DP_{(c)}[-2] \simeq \DP((P \setminus c)^{ab}, (\mG \setminus c)^{ab})$.

\subsection*{Maximal building set}

Now consider the same polymatroid $P$ with the maximal building set $\mG_{\max} = \{a,b,c,ab, \hat{1}\}$.
The polymatroid $P(a)$ relative to the maximal building set is shown in \Cref{fig:Hasse_P_gmax} and the groundset $E(a)$ is $\{b,c,ab\}$.
This polymatroid $P(a)$ associated with $\mG_{\max}$ is different from the polymatroid $P(a)$ defined from the minimal building set $\mG$ (shown in \Cref{fig:P(a)}).

\begin{figure}
    \centering
    \begin{tikzpicture}
         \node (top) at (0,3) {$\hat{1}$};
         \node (bottom) at (0,0) {$\hat{0}$};
         \node (b) at (0,1) {$b$};
         \node (c) at (1,1.5) {$c$};
         \node (ab) at (-0.5,2) {$\{b,ab\}$};
         \draw (bottom) -- (b) node [midway, left] {$\scriptstyle 2$};
         \draw (b) -- (ab) node [midway, right] {$\scriptstyle 2$};
         \draw (ab) -- (top) node [midway, left] {$\scriptstyle 1$};
         \draw (c) -- (bottom) node [midway, right] {$\scriptstyle 4$};
         \draw (c) -- (top) node [midway, right] {$\scriptstyle 1$};
    \end{tikzpicture}
    \caption{The Hasse diagram of the poset of flats of $P(a)$ with maximal building set on the groundset $\{b,c,ab\}$.}
    \label{fig:Hasse_P_gmax}
\end{figure}
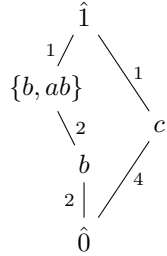

The characteristic polynomial is $\chi_P(\lambda)=\lambda^5-2 \lambda^3+1$ and the reduced one is
\[ \overline{\chi}_P(\lambda)= \lambda^4+ \lambda^3-\lambda^2-\lambda-1.\]
We have $\alpha= -x_{\hat{1}}$, $\beta=x_a+x_b+x_c+x_{ab}+x_{\hat{1}}$ and $\deg(\alpha^4)=1$, $\deg(\alpha^3 \beta)=-1$, $\deg(\alpha^2 \beta^2)=-1$, $\deg(\alpha \beta^3)=1$, and $\deg(\beta^4)=-1$.
\bibliographystyle{amsalpha}
\bibliography{bibv2}

\end{document}